\documentclass[11pt]{amsart}
\usepackage{amsmath}
\usepackage{amssymb}
\usepackage{amsfonts}
\usepackage{amsthm}
\usepackage{enumerate}
\usepackage[mathscr]{eucal}
\usepackage{amsthm}
\usepackage{amscd}

\newtheorem{theorem}{Theorem}[section]
\newtheorem{proposition}[theorem]{Proposition}

\newtheorem{corollary}[theorem]{Corollary}
\newtheorem{lemma}[theorem]{Lemma}

\theoremstyle{definition}

\numberwithin{theorem}{section}
\numberwithin{definition}{section}

\renewcommand{\wp}{\upsilon}
\def\ecc{\frak e}
\def\scr{{s_{{\text{\rm cr}}}}}
\def\trnorm{|\!|\!|}
\def\Enear{{\mathcal E^{\text{\rm near}}} }
\def\Efar{{\mathcal E^{\text{\rm far}}} }
\def\Ennear{{\mathcal E_n^{\text{\rm near}}} }
\def\Enfar{{\mathcal E_n^{\text{\rm far}}} }
\def\Ellip{\text{{\rm Ell}}}
\def\Mod{\text{{\rm Mod}}}
\def\im{\imath}

\def\prodstar{{\prod_{i=1,2}}^{\!\!*}}
\def\prodstartext{{\prod^{*}}_{\!\!i=1,2}}

\def\intslash{\rlap{\kern  .32em $\mspace {.5mu}\backslash$ }\int}
\def\inn#1#2{\langle#1,#2\rangle}
\def\biginn#1#2{\big\langle#1,#2\big\rangle}

\newcommand{\ci}[1]{_{{}_{\!\scriptstyle{#1}}}}

\newcommand{\Be}{\begin{equation}}
\newcommand{\Ee}{\end{equation}}

\newcommand{\Bm}{\begin{multline}}
\newcommand{\Em}{\end{multline}}
\newcommand{\Bea}{\begin{eqnarray}}
\newcommand{\Eea}{\end{eqnarray}}
\newcommand{\Beas}{\begin{eqnarray*}}
\newcommand{\Eeas}{\end{eqnarray*}}
\newcommand{\Benu}{\begin{enumerate}}
\newcommand{\Eenu}{\end{enumerate}}
\newcommand{\Bi}{\begin{itemize}}
\newcommand{\Ei}{\end{itemize}}

\newcommand{\R}{\mathbb{R}}
\newcommand{\Z}{\mathbb{Z}}

\def\intslash{\rlap{\kern  .32em $\mspace {.5mu}\backslash$ }\int}
\def\qsl{{\rlap{\kern  .32em $\mspace {.5mu}\backslash$ }\int_{Q}}}
\def\Re{\operatorname{Re\,}}

\def\vth{\vartheta}

\def\R{\mathbb R}

\def\N{\mathbb N}
\def\sB{{\mathscr{B}}}

\def\emph#1{{\it #1 }}

\def\eg{{\it e.g. }}
\def\cf{{\it cf}}

\def\dist{{\text{\rm dist}}}

\def\supp{{\text{\rm supp}\,}}

\def\inn#1#2{\langle#1,#2\rangle}
\def\biginn#1#2{\big\langle#1,#2\big\rangle}

\def\noi{\noindent}

\def\card{\text{\rm card}}
\def\lc{\lesssim}
\def\gc{\gtrsim}

\def\eps{\varepsilon}

\def\la{\lambda}

\def\vphi{\varphi}
             
\def\om{\omega}

\def\fA{{\mathfrak {A}}}
\def\fB{{\mathfrak {B}}}

\def\fM{{\mathfrak {M}}}

\def\fQ{{\mathfrak {Q}}}

\def\fS{{\mathfrak {S}}}

\def\fV{{\mathfrak {V}}}
\def\fW{{\mathfrak {W}}}

\def\bbN{{\mathbb {N}}}

\def\bbR{{\mathbb {R}}}

\def\bbZ{{\mathbb {Z}}}

\def\cA{{\mathcal {A}}}
\def\cB{{\mathcal {B}}}

\def\cD{{\mathcal {D}}}
\def\cE{{\mathcal {E}}}
\def\cF{{\mathcal {F}}}
\def\cG{{\mathcal {G}}}

\def\cK{{\mathcal {K}}}
\def\cL{{\mathcal {L}}}
\def\cM{{\mathcal {M}}}
\def\cN{{\mathcal {N}}}

\def\cP{{\mathcal {P}}}
\def\cQ{{\mathcal {Q}}}
\def\cR{{\mathcal {R}}}
\def\cS{{\mathcal {S}}}
\def\cT{{\mathcal {T}}}
\def\cU{{\mathcal {U}}}
\def\cV{{\mathcal {V}}}
\def\cW{{\mathcal {W}}}

\def\Re{\operatorname{Re\,} }

\begin{document}

\newcommand{\NN}{\mathbb{N}}

\newcommand{\B}{\mathbb{B}}
\newcommand{\SSS}{\mathbb{S}}


\title[Improved  bounds for Stein's square functions]
{Improved  bounds for Stein's square functions}

\thanks{Supported in part by NRF grant 2009-0072531 (Korea),
MICINN grant MTM2010-16518 (Spain) and NSF grant 0652890 (USA)}
\date{Revised version for PLMS, April 24, 2011}

\author[S. Lee \ \ \ \ K. Rogers \ \ \ \  A. Seeger]{Sanghyuk Lee \ \ \
Keith M. Rogers \ \ \ Andreas Seeger}

\address{Sanghyuk Lee\\ School of Mathematical Sciences,
Seoul National University, Seoul 151-742, Korea}
\email{shlee@math.snu.ac.kr}

\address{Keith Rogers \\
Instituto de Ciencias Matematicas
CSIC-UAM-UC3M-UCM \\ Madrid 28049, Spain} \email{keith.rogers@icmat.es}

\address{Andreas Seeger \\ Department of Mathematics 
\\ University of Wisconsin--Madison\\
480 Lincoln Drive\\
Madison, WI, 53706, USA} \email{seeger@math.wisc.edu}

\begin{abstract} We prove a weighted norm inequality for the maximal Bochner-Riesz operator and the associated square-function. This yields new
$L^p(\Bbb R^d)$ bounds on classes of radial Fourier multipliers
 for $p\ge 2+4/d$ with $d\ge 2$,
as well as space-time regularity results for the wave and
Schr\"odinger
equations.
 \end{abstract}

\maketitle

\section{Introduction}\label{intro}
Consider the  Bochner-Riesz means of order $\alpha$
defined for
Schwartz functions $f\in\cS(\bbR^d)$ by
$$ \cR^\alpha_t f(x)= \frac{1}{(2\pi)^{d}}\int_{|\xi|\le t}
\Big(1-\frac{|\xi|^2}{t^2}\Big)^\alpha\widehat f(\xi)\,e^{\im
\inn{x}{\xi}} d\xi,$$ where  $\widehat f (\xi)=\int f(y)\,e^{-\im
\inn{y}{\xi}}dy$.
In connection with questions regarding
almost everywhere summability,
and in analogy to classical Littlewood-Paley functions
for Poisson-integrals,
Stein~\cite{stein58} introduced  a  square function
defined by
$$G^\alpha\! f(x)=
\Big(\int_0^\infty \Big|\frac{\partial}{\partial t} \cR^\alpha_t f(x)
\Big|^2 t\,dt
\Big)^{1/2}.
$$
One is  interested in the inequality
$\|G^\alpha \!f\|_p\lc \|f\|_p$,
where
$A\lc B$
denotes $A\le C B$ with an unspecified constant
independent of $f$.
As $t\partial_t\widehat{\cR^\alpha_t f}(\xi)=2\alpha |\xi|^2/t^2
(1-|\xi|^2/t^2)^{\alpha-1}_+\widehat f(\xi)\,,$
one can
consider the $L^p$ problem  as a question regarding the boundedness of
a  vector valued singular integral operator involving Riesz means of order
 $\alpha-1$. It is known that $L^p$ boundedness for  $1<p\le 2$ holds
 if and only if  $\alpha>d(1/p-1/2)+1/2$ (see~\cite{su1},~\cite{igku}),
however
 the problem is more interesting in  the range $p>2$ for which
the condition $\alpha> \max\{1/2, d(1/2-1/p)\}$ is known to be necessary
and  conjectured to be  also sufficient.
For $d=1$ many proofs of the conjecture are known,
for one of them see \cite{su2}.  The conjecture in two dimensions was
proven by Carbery~\cite{ca},
 and partial results  for $p> \frac{2(d+1)}{d-1}$, $d\ge 3$,  are in~\cite{ch},
\cite{se0}. Here we improve on  the range in  dimensions $d\ge 3$.
\begin{theorem}
\label{steinsquarelp}
Let $d\ge 2$ and $p\in[\frac{2(d+2)}{d},\infty)$.
Then
\begin{equation*}\label{Lpsteinsquarefct}
\big\|G^\alpha \!f\|_p\lc \|f\|_p, \quad \text{ $\alpha>d\Big(\frac 12-\frac 1p\Big)$.}
\end{equation*}
\end{theorem}
As in the related work by the first author~\cite{le}
our  main tool  will be Tao's bilinear estimate~\cite{ta}
 for the adjoint of the Fourier restriction operator. The   square function
 result
implies  the currently known sharp $L^p$
estimates for the
maximal Bochner-Riesz operator
obtained by Carbery  ~\cite{ca} in two dimension  and by the first author ~\cite{le} in higher dimensions; however,
as pointed out in~\cite{le}, somewhat weaker  estimates are
already enough to bound  the maximal  function. More precisely,
for a   compact $t$-interval $I\subset  (0,\infty)$, the  estimate in
\cite{le} could be formulated as a  variational $L^p(V_4(I))$ inequality for Riesz means of order
 $\la> \la_{\text{crit}}:=d(1/2-1/p)-1/2$,
or a slightly better
 regularity  result involving the Sobolev space $L^4_{1/4+\eps}(I)$.

The  $L^p$-estimate  for the square function
 is  significant for various reasons.
Firstly,
it yields  regularity results for wave and Schr\"odinger operators
which  will be discussed  below. Secondly,
  for compact $I\subset (0,\infty)$
it implies   $L^p(V_2(I))$ or
$L^p(L^2_{1/2+\eps}(I))$ results for  Bochner-Riesz
means of order $\la>\la_{\text{crit}}$.
Thirdly, the $L^p$-result for $G^\alpha$ implies
 boundedness results  for maximal operators associated with
 more general classes of radial Fourier multipliers as in
\cite{caesc},~\cite{dt}, and  finally,
 an inequality by Carbery, Gasper and Trebels
\cite{cgt} relating radial multipliers and  $G^\alpha$
 yields the following  sharp $L^p\to L^p$ boundedness
result of H\"ormander-Mikhlin type.
\begin{corollary}\label{Lpmultiplierresults}
Let $d\ge 2$ and
$p\in(1, \frac{2(d+2)}{d+4}]\cup [\frac{2(d+2)}{d},\infty)$.
%
Let $\varphi$ be a
nontrivial $C^\infty$ function compactly supported in
$(0,\infty)$. Then
$$
\sup_{f\in\cS:\, \|f\|_p\le 1}
 \big\|\cF^{-1}[ m(|\,\cdot\,|) \widehat f \,]\big\|_p
  \lc\, \sup_{t>0} \|\varphi\,
 m(t\,\cdot\,)\|_{L^2_{\alpha}(\bbR)},\quad \alpha> d\Big|\frac 1p-\frac 12\Big|.
$$

\end{corollary}

\subsubsection*{\bf Weighted norm inequalities}
More information about $G^\alpha$ can be obtained by considering
an $L^2$ weighted norm inequality which involves  a
``universal'' maximal  operator $\fV_q$ acting on the weights, and
 which is strong enough to imply the above $L^p$ estimates.
The operator $\fV_q$  is a maximal operator which
is bounded on $L^r$ for $q<r\le \infty$,
and the  $L^p$ estimates  for $G^\alpha$ in Theorem
\ref{steinsquarelp} can be deduced after an additional interpolation
 if we take  $q<(p/2)'$. An informal discussion of the definition of the
weight operator is given below.

\begin{theorem}\label{rieszweighted}
 Let $d\ge 2$ and $q\in (1,\frac{d+2}{2})$.
Then there is  an operator
$\fV_q$  which is
bounded on $L^r$ for $q<r\le \infty$, such that
\begin{equation}\label{labelweightG}
\int_{\bbR^d}
  \big|G^\alpha f(x)\big|^2 w(x)\,dx
\lc \int |f(x)|^2\, \frak V_q w(x)\, dx, \quad \alpha>\frac{d}{2q}\,.
\end{equation}
Moreover,
\begin{equation}\label{labelweightriesz}
\int_{\bbR^d}
 \sup_{t>0}  \big|\mathcal R^\lambda_t  f(x)\big|^2 w(x)\,dx
\lc \int |f(x)|^2\, \frak V_q w(x)\, dx, \quad
\lambda>\frac{d-q}{2q}\,.
\end{equation}
\end{theorem}

The weighted inequalities
\eqref{labelweightG} and \eqref{labelweightriesz}
are motivated by  one of Stein's
problems in ~\cite{st}. It was asked  whether the operator  defining the
weight on the right hand side of \eqref{labelweightriesz}
could  be chosen to be   a  Nikodym maximal
operator (see also C\'ordoba \cite{co0} for a related question).
This seems currently unknown.
For the range $q\in(1,\frac{d+1}2]$,
Christ~\cite{ch} proved the weighted inequality with the  simple
weight $(M|w|^q)^{1/q}$, where  $M$ denotes the Hardy--Littlewood
maximal operator. In two dimensions, Carbery~\cite{ca2} proved a
weighted inequality with an  operator $W_2$ in place of $\fV_2$,
such that $W_2$ is bounded on $L^r(\bbR^2)$ for $r\in (2,4]$.
 The extension of that
result  with the weight operator  bounded for
$r\in (2,\infty]$ was established by Carbery
and the third author~\cite{case}.

We now give an informal description of the weight operator $\fV_q$
 and refer to \S\ref{defofW}
for the  precise
description of a closely related  operator $\fW_q$
(with a more technical definition)
which  will be of weak
type $(q,q)$ and  can
be used instead of
$\fV_q$. The reader may then check that $\fV_q$ satisfies
$\fV_{q}\gc \fW_{q}$ and is still bounded on $L^{r}$ for $r>q$.
In \S\ref{defofW}  we
shall also  prove  a  refinement of Theorem \ref{rieszweighted}
which will be significant for  endpoint bounds
such as Theorem \ref{wavey} below.

 The weight
$\fV_qw$ involves the sum of two maximal functions associated with
tubes of large eccentricity;
$$ \fV_q w=
\Big(M \big[\sup_{\ecc\ge 1} {\ecc}^{-2(\frac dq-1)} \big(M
V_{\ecc}w \, +\, [\log(2+\ecc)]^{2} \sup_{l\in \bbZ} M
V^{\text{main}}_{\ecc,l,q}(P_lw)\big)\big]^{1+\eps}\Big)^{\frac{1}{1+\eps}}\,,$$
with some $\eps>0$.
Here  $P_l w$ is a standard dyadic frequency cutoff localizing
$\widehat w$ to frequencies of size $\approx 2^l$. The first summand
is similar to the  standard Nikodym maximal function, but much
better behaved due to the small damping factor. For $\ecc\ge 1$, the
function $V_{\ecc} w$ is the usual maximal function  associated with
the tubes (or cylinders) which are centered at the origin with
eccentricity (defined as the quotient  length/width) equal to
$\ecc$. The second summand involves the
 maximal function
$$V_{\ecc,l,q}^{\text{main}}g(x)=
 \sup_{\theta\in S^{d-1}}\big(M_{\theta,\ecc}[\sup_{\Psi}
|\Psi* g|^q](x)\big)^{1/q},$$ where the supremum in
 $\Psi$ ranges over  suitable classes of  $L^1$-normalized Schwartz functions associated with
tubes  of length $2^{-l}\ecc$ and width $2^{-l}$, in the direction
of $\theta$. The maximal operator $M_{\theta,\ecc}$ is associated
with the tubes in the direction $\theta$, with fixed eccentricity
$\ecc$.
This  definition is reminiscent of the
``grand maximal function'' in the theory of Hardy spaces
(\cite{fest}) as it involves  a supremum over convolutions with kernels
in a suitably normalized and rescaled
class of Schwartz functions, and a significant  gain is achieved when these kernels
are convolved with functions that have a suitable cancellation
property  (such as $P_l w$).

Concerning the boundedness properties of  $\fV_q$, for the first summand
 we shall
only  need a standard and non-optimal  bound $\|V_\ecc\|_{L^q\to
L^q}=O(\ecc^{a})$ with $a>(d-2)/q$, for $q\ge 2$, and the small
factor $\ecc^{2-2d/q}$ helps to get   the claimed bound. For the
main term the range cannot be improved
 as the maximal function involves powers $|\Phi*w|^q$. Here  the
terms with small  eccentricity ($\approx 1$)  contribute most, and  in order to establish the proper bounds for terms with
large eccentricity  one uses cancellation, namely the  annular support property
of $\widehat {P_lw}$.




\subsubsection*{\bf Wave and Schr\"odinger operators.}
One can apply $L^p$ bounds for variants of the Bochner-Riesz square-function
to  obtain regularity results for spherical means
and  solutions of the
wave and Schr\"odinger equations.
This application is suggested  by a theorem of
 Kaneko and Sunouchi~\cite{ks} relating $G^\alpha$ to another
square function which was   introduced by Stein in his study of spherical maximal operators (\cf.~\cite{stsph}).
Define  the spherical mean of order
$\beta>0$
by
$$\cA^\beta_t f(x) = \frac{\Gamma(\frac{d+2\beta}{2})}{\pi^{d/2}\Gamma(\beta)}
\int_{|y|\le t}\frac{1}{t^{d}} \Big(1-\frac{|y|^2}{t^2}\Big)^{\beta-1} f(x-y)\,dy;$$
 for smaller values of $\beta$ the definition can be extended
by analytic
continuation. In~\cite{ks}  an application of Plancherel's theorem with
 respect to $t$  is used to show  that $G^\alpha$ is
pointwise
equivalent with a square function generated by  spherical means, namely, for $\alpha>0$,
$$G^{\alpha} f(x) \approx  \Big(\int_0^\infty \Big| \frac{\partial}{\partial t}
\cA^{\alpha- \frac{d-2}{2}}_t \!f(x)\Big|^2 t\,dt\Big)^{1/2}, \quad
$$
for  all Schwartz functions $f$.
Here we shall  not use this equivalence explicitly but prove a closely
related sharp $L^p(L^2)$ regularity result
for solutions of the wave and Schr\"odinger
 equations with initial data in $L^p$
Sobolev spaces.
In order to formulate a unifying result, for $a\in(0,\infty)$, we let $U_t^af$ denote the solution
to the initial value problem
$\im\partial_tu +(-\Delta_x)^{a/2}u=0$ with  $u(\,\cdot\,,0)=f$;
\Be\label{Ua}
U_t^a f  = \exp (\im t(-\Delta)^{a/2})f.
\Ee
The case $a=2$ corresponds to the Schr\"odinger equation and the
case $a=1$ to a  wave equation.

\begin{theorem}\label{wavey} Let   $d\ge 2$, $p\in(\frac{2(d+2)}{d},\infty)$, $a\in(0,\infty)$, and let $I$ be a compact time interval. Then
\begin{equation}\label{psa}
\Big\|\Big(\int_I|U^a_t f|^2 \,dt\Big)^{1/2}\Big\|_{p}
\lc \|f\|\ci{L^p_{s}},\quad
\frac sa =d\Big(\frac{1}{2}-\frac{1}{p}\Big)-\frac{1}{2}.
\end{equation}
\end{theorem}
\noindent In fact
 this holds for initial data in the Besov space $B^{p}_{s,p} $ (which contains $L^p_s$ for  $p\ge 2$).

One can also consider the same regularity problem
in the mixed norm space $L^p(L^q(I))$ with $q\in(2,\infty]$.
For this range the analogy between the wave and Schr\"odinger equation breaks down
(and  some endpoint versions of the deeper `local smoothing'
 result for the wave equation are currently available only in  four and higher
dimensions, \cf.~\cite{honase}).
However, for  $a\in (0,1)\cup(1,\infty)$ and $d\ge 2$, one can deduce sharp
$B^{p}_{s,p}\to L^p(L^q(I))$
estimates with
$s =ad(1/2-1/p)-a/q$,
 in the range
$p\in(2+ 4/d,\infty)$.
This follows
from a combination of Theorem~\ref{wavey} and the  result in  Appendix~\ref{localtoglobal}.
Moreover,
one can, for a limited
range of $q$, obtain further estimates for $p>2+4/(d+1)$ and $d\ge 1$,
 essentially by interpolation with results in
\cite{rose}; in dimensions $d\ge 2$
this currently requires the restriction $a>1$.
 These $L^p(L^q)$-estimates are stated in \S\ref{LpLqsect},
and a further, more substantial improvement for $d=2$
will be considered in
\cite{lerose2}.

\subsubsection*{Remark} After the first version of this paper was submitted
for publication,
Bourgain and
Guth posted a preprint \cite{bogu}  containing  very substantial
 improvements on the
$L^p$  boundedness for oscillatory integrals related to the Fourier
restriction problem, with implications for Bochner-Riesz
multipliers. It would be of great interest to investigate the impact
of their methods on Stein's square-function, weighted norm
inequalities, and other issues discussed in this paper.

\subsubsection*{This paper}
In \S\ref{defofW} we formulate a more precise weighted inequality (Theorem~\ref{weightL2thm}), and give the  definitions and boundedness properties
 of suitable weight operators.
 In \S\ref{biladjrestr} we prove some $L^2\to L^p$
estimates for radial convolution operators and prepare for the proof
of the  weighted inequalities. These are established in
\S\ref{proofweighted}. In \S\ref{wschrsect} we prove $L^p(L^2)$
estimates for wave and Schr\"odinger equations and in
\S\ref{LpLqsect} we discuss some $L^p(L^q)$ bounds for $q>2$.
Appendix~\ref{localtoglobal} contains  auxiliary results on
combining inequalities for frequency localized operators.

\subsubsection*{Some notational references}
For two nonnegative quantities $A$, $B$ the notation  $A\lc B$, or
 $B\gc A$,  is used for $A\le CB$, with  some unspecified positive constant $C$.
We also use $A\approx B$ to indicate that $A\lc B$ and $B\lc A$.
To avoid unwieldy formulas we will
sometimes shorten the notation for   products involving a complex conjugate and use, given two complex  terms  $\cE_1$ and $\cE_2$,   the expression
$\prodstartext[\cE_i] = \cE_1\overline{\cE_2}.$
For convolution operators given by Fourier multipliers $a(\xi)$ we occasionally use
the symbol notation
$a(D)f:=  \cF^{-1}[a\widehat f \,]$, where $\cF^{-1}$ denotes the inverse Fourier transform.
By $P_n$, $\widetilde P_n$  we denote dyadic frequency cutoff operators
which  localize to
frequencies of size $\approx 2^n$, so that $P_n\widetilde P_n=P_n$,  see
\S\ref{defofW} for the precise definition.

\section{A stronger weighted norm inequality}\label{defofW}
We formulate a  weighted  norm inequality for a square function
generated  by thin pieces of the Bochner-Riesz multiplier.
To fix notation,
let $\phi$ be a Schwartz function supported in $(1/2,2)$ with the property that
\Be\label{schwartzestimates}
|\phi^{(\nu)}(t)|\le 1 , \quad \nu=0,\dots, d+2. \Ee
Let $0<\delta< 1/2$ and define the convolution operator $S^{\delta}_t\equiv S^{\delta,\phi}_t$ by
\Be \label{Sdeltat}\widehat{S^\delta_t f}(\xi)= \phi\big(\delta^{-1}\big(1-\tfrac{|\xi|^2}{t^2}\big)\big)\widehat f(\xi).
\Ee
Assuming \eqref{schwartzestimates} we shall usually drop the
superscript $\phi$,
as our estimates will be understood to be uniform in  $\phi$.

\begin{theorem}\label{weightL2thm}
Let $d\ge 2$ and $q\in(1,\frac{d+2}{2})$.
For  $0<\delta< 1/2$,
there are operators
$\fW_{q,\delta}$
defined on $L^q+L^\infty$, so that
the weighted norm inequality
\Be\label{weightboundSdelta}
\int_{\R^d}\int_{0}^\infty|S^\delta_tf(x)|^2\frac{dt}{t}\, w(x)\,dx
\,\lc\,\delta^{2-d/q}\int_{\R^d} |f(x)|^2\,\fW_{q,\delta}
w(x)\,dx
\Ee
holds for all $w\in L^q+L^\infty$
and the  operators $\fW_{q,\delta}$ satisfy the following properties:

(i) The maximal operator  defined by
\begin{equation}\label{tag}\fW_qw =\sup_{0<\delta<1/2} \fW_{q,\delta} w\end{equation}
is of weak type $(q,q)$ and bounded on $L^r$ for
$q<r\le\infty$.

(ii) If $q\in[2,\frac{d+2}{2})$, then the operators $\fW_{q,\delta}$
are bounded
on $L^q$, uniformly in $\delta$.
Moreover if $q\in(1,2)$ then $\|\fW_{q,\delta}\|_{L^q\to L^q} \lc [\log(\tfrac 1\delta)]^{\frac 1q-\frac 12}$.

\end{theorem}

 We shall also consider local versions of
\eqref{weightboundSdelta} with the $t$-integral extended over a dyadic interval and for which  the $L^q$ bounds of the
 corresponding weight operators  are independent of $\delta$ for all
$q\in (1,\frac{d+2}{2})$, see Theorem \ref{Tdeltat} below.

To deduce  inequality \eqref{labelweightG} with $\fW_q$ in place of
$\fV_q$
one splits  the multiplier
into  a part  near the origin and a part near the unit sphere.
The part near the origin is  dealt with by
 the
standard weighted norm inequality for singular integrals in
\cite{cofe}. One then  decomposes the  part near the unit sphere
into smooth multipliers supported on thin annuli of width
$\delta=2^{-j}$, applies  Theorem \ref{weightL2thm}, and sums a
geometric series. The maximal inequality  \eqref{labelweightriesz}
follows from \eqref{labelweightG}  by well-known arguments in
(\cite[\S VII.5]{stwe}) together with a weighted norm inequality for
the Hardy--Littlewood maximal function (\cite{fest0}). If we take
$q=(p/2)'$
then by  duality
and  an
 application of the Marcinkiewicz interpolation theorem
one obtains
Theorem~\ref{steinsquarelp}  for $p>2+4/d$. Interpolation with
an $L^2$ inequality yields the result  also for $p=2+4/d$.
Theorem \ref{weightL2thm} also implies a sharp $L^p$
result for the square-functions generated by $S^\delta_t$ which is stated
 in Corollary \ref{Tdeltatcor}  below.

%

\subsection*
{Definition of $\fW_{q,\delta}$.} We assume throughout this section that
 $d\ge 2$.
The definition of $\fW_{q,\delta} w$ in  \eqref{fWqdel}  involves a
suitably damped
Nikodym maximal function  and another (more important)
 maximal function acting on functions with
Fourier transform supported away from the origin.

Let $0<\delta_\circ<1$, let $\theta$ be a unit vector in $\bbR^d$  and let
$$R_{\delta_\circ,t}^\theta  =
\{y\in \bbR^d: |\inn y \theta|\le t, |y-\inn y\theta \theta |\le t\delta_\circ\}.$$
Then the Nikodym maximal function associated with tubes of
eccentricity $\delta_\circ$ is defined by
\begin{equation}\label{nikodym} \fM_{\delta_\circ} g(x)
=\sup_{\theta\in S^{d-1}}\sup_{t>0}
\frac{1}{|R_{\delta_\circ,t}^\theta|}
\int_{R_{\delta_\circ,t}^\theta} |g(x+y)| dy.
\end{equation}


Now we describe our second maximal operator.
 Let  $\cN \ge d+3$ be a large positive integer
and let
$\cS(\cN)$ be  the set of all Schwartz functions $\psi$
  for which
\Be \label{SNdef}
\trnorm \psi\trnorm_{\cN} := \max_{|\alpha|\le \cN} \sup_x (1+|x|)^{\cN} |\partial^\alpha_x \psi(x)| \,\le\,1.
\Ee
The number $\cN$ will be fixed throughout the paper and constants in
inequalities will  depend on $\cN$ (one may want to choose $\cN=d+3$).

For  $j\ge 0$ and $\theta\in S^{d-1}$ let $\ell_{\theta,j}$ be the unique linear transformation defined by
$$
\ell_{\theta,j} (\theta)= 2^j\theta, \qquad
\ell_{\theta,j} (y)= y\ \text{ if }\ \inn{\theta}{y}=0.
$$
Then $\det \ell_{\theta,j}= 2^j$.
Let $\cS^{\theta,j}_n$ be the set of all
$\Psi$ for which
$2^{-nd}2^j \Psi (\ell_{\theta,j} 2^{-n} \,\cdot\,)$ belongs to $\cS(\cN)$.
Typical examples of  functions in $\cS^{\theta,j}_n$ are $L^1$ normalized
 bump functions  essentially supported on a tube with direction
$\theta$,  length $2^{j-n}$ and width $2^{-n}$. We define a maximal function which involves convolutions with
 $\Psi$ in the
classes   $\cS^{\theta,j}_n$  and in our application it is crucial that
these  convolutions
will be acting on functions with cancellation, namely with frequency support in annuli.
Let $$\cM^{\theta,j}_n g(x)= \sup_{\Psi\in \cS^{\theta,j}_n}|\Psi* g(x)|$$
and set, for $\tau>0$,
\Be \label{cKkernel}
\cK_\tau^{\theta,j}(x)
: = \frac{2^{-j}\tau^d}
 {(1+| \ell_{\theta,j}^{-1}(\tau x)|)^{d+1}}.
\Ee
For future reference note that
\Be \label{Psiestim}
|\Psi(x)|
\,\lc \,
\frac{2^{kd}2^{-j(d+1)}}
 {(1+ 2^{k-2j}|\inn x\theta| +2^{k-j}|x-\inn x\theta \theta|)^{d+3}}
\lc\cK_{2^{k-j}}^{\theta,j}(x)
\,\,\text{ if  $\Psi\in \cS^{\theta,j}_{k-j}\,.$}
\Ee

We use the   dyadic frequency cutoff
operator  $P_n$  defined by
\Be\label{Pndef} \widehat {P_n f}(\xi)=\chi(2^{-n}|\xi|)\widehat f(\xi)\Ee
where $\chi\in C^\infty$ is nonnegative and  supported in $(5/8, 15/8)$ so that
$\sum_{k\in \bbZ} \chi(2^{-k}t)=1$ for all $t>0$.
Set
\begin{align}
\label{cWqdefeqk}
\cW_{q,\delta,k}^j g(x)&:=
\Big(\sup_{\theta} \cK^{\theta,j}_{2^{k+j}\delta} *\big|\cM^{\theta,j}_{k-j}
P_{k-j} g|^q(x) \Big)^{1/q}\, ,
\\
\label{cWqdefeq}\cW_{q,\delta}^j g(x)&:=\sup_{k\in \bbZ} \cW_{q,\delta,k}^j g(x)\, .
\end{align}
Next,  fix $s\in (1,q)$  and define
the maximal operator $\fW_{q,\delta}$ by
\Be \label{fWqdel}\fW_{q,\delta} w\,:=\,\sum_{1\le 2^{2j}<\delta^{-1}}
2^{-2j(\frac dq -1)}\big(
M|M\!\circ\!\cW_{q,\delta}^j w|^s\big)^{1/s}
+\delta^{\frac{d}{q} -1} \big(
M|M\!\circ\!\fM_{\sqrt \delta} w|^s\big)^{1/s}
\Ee
where $M$ is the Hardy--Littlewood maximal operator.
We also recall from the statement of Theorem~\ref{weightL2thm} the definition
$\fW_qw=
\sup_{0<\delta< 1/2}\fW_{q,\delta}w.$



\subsection*{Boundedness of the weight operators}
In the proofs we will frequently use a dyadic frequency cutoff
 $\widetilde P_n$  which reproduces $P_n$ and is similarly defined.
That is to say, $\widehat {\widetilde P_n f}(\xi)=\widetilde \chi(2^{-n}|\xi|)\widehat f(\xi)$
where $\widetilde \chi$  is supported in $(1/2,2)$ and has the property
$\widetilde \chi(s)=1$ for $s\in [5/8, 15/8]$.
Then $P_n\widetilde P_n=P_n$.

It is obvious that the  operators $\cW^j_{q,\delta,k}$,
$\cW^j_{q,\delta}$, $\fM_{\sqrt\delta}$ are bounded on $L^\infty$.
For the $L^q$ boundedness   we state

\begin{proposition}  \label{maximaltheorem} %
(i) For $1\le q\le \infty$, the operator $\cW^j_{q, \delta,k}$
%
%
 satisfies
$$\sup_{\delta,k} \sup_{\|w\|_q\le 1} \|\cW^j_{q,\delta, k} w\|_q\lc
\begin{cases} 2^{j\frac{d-2}q} \quad&\text{ if }\ 2\le q\le\infty,
\\
2^{j(\frac dq -1)} \quad&\text{ if }\ 1\le q< 2.
\end{cases}
$$

\noindent (ii) For $1<q\le \infty$, the operator $\cW^j_{q, \delta}$
satisfies
$$\sup_{\|w\|_q\le 1} \|\cW^j_{q,\delta}w\|_q\lc
\begin{cases} 2^{j\frac{d-2}q} \quad&\text{ if }\ 2\le q\le\infty,
\\
2^{j(\frac dq -1)} [\log(\frac 1\delta)]^{\frac 1q-\frac 12}\quad&\text{ if }\ 1< q< 2.
\end{cases}
$$
Moreover for $q=1$ the operator $\cW^j_{1, \delta}$  maps the Hardy space
$H^1$ to $L^1$ with operator norm $\lc
2^{j(d -1)} [\log(\frac 1\delta)]^{1/2}$.

\noindent (iii)
We also have the weak type $(q,q)$ estimate
$$
\sup_{\|w\|_q\le 1}
\big\| \sup_{0<\delta<1/2} \cW^j_{\delta,q} w\big\|_{L^{q,\infty}}
\lc \begin{cases} 2^{j\frac{d-2}q} \quad&\text{ if }\ 2\le q< \infty,
\\
2^{j(\frac dq -1)}(1+j)^{\frac 1q-\frac 12} \quad&\text{ if }\ 1< q< 2.
\end{cases}
$$
\end{proposition}
\begin{proof}[The proposition implies statements
(i) and (ii) of  Theorem \ref{weightL2thm}.] Clearly  the
operators $\fW_{q,\delta}$ and $\fW_q$ are bounded on $L^\infty$ if
$q<d$.
The $L^q$ bound for the first (main) term in
\eqref{fWqdel} is immediate from Proposition \ref{maximaltheorem}
since $\frac{d-2}{q}-2(\frac dq -1)<0$ iff $q<\frac{d+2}{2}$. For
the second term  in \eqref{fWqdel} we use standard non-endpoint
$L^q$ bounds for the Nikodym maximal operator (see \cite{co0},
\cite{co1}, \cite{bo}, \cite{mss}). Namely $\fM_{\sqrt\delta}$
is bounded on $L^q$ with operator norm $\le C_\eps
(\sqrt \delta)^{1- d/q-\eps}$ if $q<2$
and operator norm
$\le C_\eps
(\sqrt \delta)^{- (d-2)/q-\eps}$ if $q\ge 2$.
The damping factor $\delta^{-1+d/q}$ is enough to prove $L^q$ boundedness for
$q<\frac{d+2}2$. Using for example the results in \cite{wo}   this final estimate
can be significantly improved but any such  improvement  seems currently
to have  no
impact on our result, as the main contribution to the weight operator
comes from  the terms $\cW_{q,\delta}^j$.
\end{proof}

\subsubsection*{Elementary  convolution estimates}
The following simple and standard
convolution estimates will be used many times in the paper.

\begin{lemma}\label{prelptwise}
(i)   Let $H(x)=(1+|x|)^{-N}$ and let $N>d$.
Then there is $C_{d,N}>0$ so that for all $x\in \bbR^d$
\Be\sup_{t\ge 1} \sup_{0\le s\le 1} \int t^dH(t y)H(x-sy)\,dy
\le C_{d,N} H(x).\Ee

\noindent (ii) Let  $H^A(x)= |\det A| H(Ax)$. Then  $H^A*H^A(x) \le C_{d,N} H^A(x)$  for all $x\in \bbR^d$.

\noindent (iii) Let $\ell\in \bbN$,  let $h_1$, $h_2$ be kernels with
$$(1+|x|)^\ell |h_1(x)| \,+\, \sum_{|\alpha|=\ell} |\partial^{\alpha} h_2(x)|\le H(x)$$
and assume that
$\int P(x) h_1(x)\,dx=0$ for all polynomials $P$ of degree
$\le  \ell-1$.
Let $h_i^A(x) =|\det A| h_i(Ax)$.
Then for $t\ge 1$
$$\big| t^dh_1^A(t\,\cdot\,) * h_2^A (x)| \le C_{d,N,\ell} \,t^{-\ell}  H^A(x)\,.$$
\end{lemma}\begin{proof}
Let $F_{t,s}(x)=\int t^dH(t y)H(x-sy)\,dy $ then clearly $F_{t,0}(x)\lc H(x)$ and
for $|x|\le 1$ we have $F_{t,s}(x)\lc 1$ so that
the assertion holds for $|x|\le 1$.

If $|x|\ge 1$ then $\sup_{|sy|\le |x|/2}H(x-sy) \lc H(x)$ and
$\sup_{|sy|\ge 2|x|}H(x-sy) \lc H(x)$, so that
$$
\int_{|sy|\notin [\frac{|x|}{2}, 2|x|]} t^d H(ty)H(x-sy)\,dy \lc
 \|t^d H(t\,\cdot\,)\|_{1}H(x) \lc H(x),
$$ for all $t>0$.

Next, if $|sy|\approx |x|$ then
$t^d H(ty)  \approx t^d H(tx/s)$. Letting
$B_{l,s}(x)=\{y:|sy-x|\le 2^l\}$ for $l\ge 0$, we have $|B_{l,s}(x)|\lc (2^l/s)^d$
and we may estimate (using $N>d$)
\begin{multline*}
\int_{|x|/2\le |sy|\le 2|x|} H(x-sy) t^d H(ty)\,dy\\
\lc  t^{d-N}s^N|x|^{-N}\sum_{l\ge 0}
\int_{B_{l,s}(x)} 2^{-lN} dy \lc t^{d-N}s^{N-d} |x|^{-N} \lc H(x)
\end{multline*}
since $t\ge 1$, $s\le 1$ and $|x|\ge 1$.

\textit{(ii)} follows immediately by a change of variable.
Similarly for $\textit{(iii)}$ we may reduce to the case where $A$ is the identity. Then one can use  Taylor's formula and  the cancellation of $h_1$, and \textit{(i)},  to estimate
\begin{align*}
|t^dh_1(t\,\cdot\,)*h_2(x)|&\lc \int_0^1\frac{(1-s)^{\ell-1}}{(\ell-1)!} \int\big| t^dh_1(ty)
\inn{y}{\nabla }^\ell h_2 (x-sy)\big| \, dy\,ds\,
\\
&\lc t^{-\ell} \sup_{0\le s\le 1}\int t^dH(ty) H(x-sy) \, dy\,\lc t^{-\ell} H(x)\,.
\end{align*}
\end{proof}

\noindent For the kernels in \eqref{cKkernel}, Lemma
\ref{prelptwise} yields the bound
\begin{equation}\label{cKconv}\cK_{t}^{\theta,j}* \cK_{\tau}^{\theta,j}(x) \le C_d
 \cK_{\tau}^{\theta,j}(x)\,,\quad t>\tau\,.\end{equation}

\begin{proof}[\bf Proof of Proposition \ref {maximaltheorem}]
 We  shall use many times that
$\|\cK^{\theta,j}_t\|_1\lc 1$, uniformly in $\theta, j,t$. Moreover,
\Be\label{Kthetajtptwise}
\cK^{\theta,j}_t(x) \approx \cK^{\tilde \theta, \tilde j}_{\tilde t}(x)\,,
\quad |\theta-\tilde \theta|\le C 2^{-j}, \quad |j-\tilde j|\le C,
 \quad C^{-1}\le
 t/\tilde t \le C\,,
\Ee
where the implicit constants in the equivalence depend only on  $C$ and the dimension.
Next, the class $\cS^{\theta,j}_n$ is stable under
small perturbation in the sense that given $A>0$
there is a  constant $C$,   depending only on $A$, $d$ and
 the parameter $\cN$ in the definition \eqref{SNdef},
so that
\Be\label{Sthetajperturb}
\Psi\in \cS^{\theta,j}_n \,\implies \,
C^{-1}\Psi\in \cS^{\tilde \theta,\tilde j}_{\tilde n} \text{ if
$|\theta- \tilde \theta|\le A 2^{-j}$, $|j-\tilde j|\le A$,  $|n-\tilde n|\le A$.}
\Ee

In what follows we let $\Theta_j$ be a maximal $2^{-j-d} $-separated subset
 of $S^{d-1}$.
Clearly $\card(\Theta_j)\lc 2^{j(d-1)}$; moreover each $\tilde \theta\in S^{d-1}$ has distance $\le 2^{-j}$ to at least one $\theta\in \Theta_j$.

\subsubsection*{$L^q$-bounds for  $\cW_{q,\delta,k}^j$}
We replace a $\sup$ in $\theta\in \Theta_j$ by an $\ell^q$ norm and
interchange a summation and integration to
 estimate for any fixed $j$, $k$
\begin{align}
&\Big\|\Big( \sup_{\theta} \cK^{\theta,j}_{2^{k+j}\delta}
*\big|\cM^{\theta,j}_{k-j}
P_{k-j} g\big|^q
\Big)^{1/q}\Big\|_q
\notag
\\
&\lc\Big( \sum_{\theta\in \Theta_j}\iint
\sup_{|\tilde \theta-\theta|\le 2^{-j} }
\cK^{\tilde \theta,j}_{2^{k+j}\delta}(y)
\sup_{|\tilde \theta-\theta|\le 2^{-j} }\big|\cM^{\tilde\theta,j}_{k-j}
P_{k-j} g(x-y)\big|^q dy dx\Big)^{1/q}
\notag
\\
&\lc\,
\Big( \sum_{\theta\in \Theta_j}\big\|
\cM^{\theta,j}_{k-j}
P_{k-j} g\big\|_q^q\Big)^{1/q}\, .
\notag
\end{align}
For the last inequality we have used
\eqref{Kthetajtptwise} and \eqref{Sthetajperturb}.

We now need a further decomposition.
Let $\eta_0$ be supported in $(-1,1)$ so that $\eta_0(s)=1$ for $s\in (-1/2,1/2)$ and let
$$\eta_\theta^{j,k} (\xi)= \eta_0\big(\sqrt{|2^{2j-k}\inn\xi\theta|^2
+|2^{j-k}(\xi-\inn\xi\theta \theta)|^2}\big)\,.
$$
 For $m\ge 0$ define
operators $Q_{m}^{\theta,j,k}$, $P_{m}^{\theta,j,k}$ as follows.
For $m=0$ set
$$\widehat{Q_{0}^{\theta,j,k}f}(\xi)=\widehat{P_{0}^{\theta,j,k} f}(\xi) =
\eta_\theta^{j,k}(\xi)
\widehat f(\xi)$$ and, for $m\ge 1$, set
\begin{align*}
\widehat {Q_m^{\theta,j,k} f}(\xi)&= \big(
\eta_\theta^{j,k}(2^{-m}\xi)-\eta_\theta^{j,k}(2^{-m+1}\xi)\big)
\widehat f(\xi)\, ,
\\
\widehat {P_m^{\theta,j,k} f}(\xi)&=
\big(\eta_\theta^{j,k}(2^{-m}\xi)+\eta_\theta^{j,k}(2^{-m+1}\xi)\big)
\widehat f(\xi)\,.
\end{align*}
Then $ \sum_{m=0}^\infty
Q_m^{\theta,j,k}
P_m^{\theta,j,k} $ is the identity; moreover
$Q_m^{\theta,j,k}
P_{k-j} =0$ for $m>j+3$. Using the cancellation of
$Q_m^{\theta,j,k} $ it is straightforward to derive
the estimate
$$
\sup_{\Psi\in \cS^{\theta,j}_{k-j}}
|Q_m^{\theta,j,k} \Psi(x)|\le C 2^{-2m}\frac{2^{(k-j)d} 2^{-j}}
{(1+ 2^{k-2j}|\inn{x}{\theta}| +2^{k-j}|x-\inn{x}{\theta}\theta|)^{d+1}}\,
$$
from Lemma~\ref{prelptwise}
and \eqref{Psiestim};
in fact if $\cN >d+3$ in \eqref{SNdef} one can also get  a gain of
higher powers of $2^{-m}$.
Now,  for fixed $k$, $j$, $\theta$
\begin{align}\notag
\Big(\sum_{\theta\in \Theta_j}\big\|
\cM^{\theta,j}_{k-j}
P_{k-j} g\big\|_q^q\Big)^{1/q}
&\lc\Big(\sum_{\theta\in \Theta_j}\Big\|
\sup_{\Psi\in \cS^{\theta,j}_{k-j}} \Big|\sum_{0\le m\le j+3} ( Q^{\theta,j,k}_m\Psi) *
(P^{\theta,j,k}_m P_{k-j} g)\Big|\Big\|_q^q\Big)^{1/q}
\\ &\lc \sum_{0\le m\le j+3} 2^{-2m}
\Big(\sum_{\theta\in \Theta_j}\big\|
P^{\theta,j,k}_m
P_{k-j}  g\big\|_q^q\Big)^{1/q}
\label{fixedkjtheta}
\end{align}
and the desired bound follows when we establish  the estimate
\Be\label{overlap}
\Big(\sum_{\theta\in \Theta_j}\Big\|
P^{\theta,j,k}_m
P_{k-j} g
\Big\|_q^q\Big)^{1/q} \lc \max\{ 2^{j\frac{d-2}q}, 2^{j(\frac dq -1)}\}\,2^{m/q} \|g\|_q, \quad 1\le q\le \infty,
\Ee
with the usual modification for $q=\infty$.
To prove  \eqref{overlap} we notice that by interpolation we only
have to verify the cases $q=1,2,\infty$. The cases $q=\infty$ and $q=1$ are immediate
since the operators
$P^{\theta,j,k}_m $ and $P_{k-j}$ have convolution kernels with uniformly bounded
 $L^1$ norms and since $\card(\Theta_j)\lc 2^{j(d-1)}$.
For $q=2$ we use Plancherel's theorem.
Write $\cF[P^{\theta,j,k}_m
P_{k-j} g] = a_{j,k,\theta,m}
\widehat g$. Then the multiplier  $a_{j,k,\theta,m}$
is supported in
$$
\{\,\xi:\ C^{-1}2^{k-j-3}\le|\xi|\le
C2^{k-j+2},\quad |\inn\xi\theta|\le C2^{k-2j+m}\,\}.
$$
Since
 $\Theta_j$ is a  $2^{-j-d}$-separated set
we see that for fixed $j,k,m$  every $\xi\in \bbR^d$  is contained
in no more than $C_d 2^{m+j(d-2)}$ of the sets
$\supp(a_{j,k,\theta,m})$, which implies the bound \eqref{overlap}.

\subsubsection*{$L^q$-bounds for  $\cW_{q,\delta}^j$}
We argue as above and now replace the sup in $k\in \bbZ$, $\theta\in \Theta_j$
by an $\ell^q$ norm.
This yields
\Be\label{nodelta}
\Big\|\Big(\sup_k \sup_{\theta} \cK^{\theta,j}_{2^{k+j}\delta}
*\big|\cM^{\theta,j}_{k-j}
P_{k-j} g(x)\big|^q
\Big)^{1/q}\Big\|_q
\lc\Big(\sum_k \sum_{\theta\in \Theta_j}\big\|
\cM^{\theta,j}_{k-j}
P_{k-j} g\big\|_q^q\Big)^{1/q}\,.
\Ee
A combination of \eqref{fixedkjtheta} and \eqref{overlap} together with the use of the reproducing Littlewood-Paley cutoff $\widetilde P_{k-j}$ yields
\begin{align*}
&\Big(\sum_k \sum_{\theta\in \Theta_j}\Big\|
\cM^{\tilde\theta,j}_{k-j}
P_{k-j} g\Big\|_q^q\Big)^{1/q}
\\&\lc
\sum_{0\le m\le j+3} 2^{-2m}
\Big(\sum_k\sum_{\theta\in \Theta_j}\big\|
P^{\theta,j,k}_m
P_{k-j} \widetilde P_{k-j} w\big\|_q^q\Big)^{1/q}
\\
&\lc
\sum_{0\le m\le j+3} 2^{-m(2-\frac 1q)}
\max\{ 2^{j\frac{d-2}q}, 2^{j(\frac dq -1)}\}
\Big(\sum_k\big\|\widetilde P_{k-j} w\big\|_q^q\Big)^{1/q}\,
\end{align*}
and for $q\ge 2$ this is $\lc 2^{j(d-2)/q}\|w\|_q$.

Now consider the case $q<2$.
By a standard linearization combined with an analytic family argument the claimed $L^q$  estimates can be deduced from the   Hardy-space estimate
 and the already proven $L^2$ estimate. We omit the details of the interpolation argument (\cf. \cite{fest}).
For the Hardy-space estimate
we need to prove an inequality for an $L^2$-atom,
{\it i.e.}
 an  $L^2$ function   $g_\rho$
supported on a ball $\{x:|x|\le \rho\}$ with
$\|g_\rho\|_2\le \rho^{-d/2}$ and $\int g_\rho\,dx=0$.
It suffices to verify
\Be\label{atomest}
\|\cW_{1, \delta}^j g_\rho\|_1\lc 2^{j(d -1)} [\log(\tfrac {1}{\delta})]^{1/2}\,.
\Ee
By the Schwarz inequality we have
$$\cW_{1,\delta}^j g(x) \lc \cW_{2,\delta}^j g(x)$$ and thus by the above $L^2$ estimates
$$
\|\cW_{1,\delta}^j g_\rho\|_{L^1(|x|\le 2 \rho)}
\lc \rho^{d/2}
\|\cW_{2,\delta}^j g_\rho\|_{2} \lc 2^{j(\frac d2 -1)} \rho^{d/2}\|g_\rho\|_2
\lc 2^{j(\frac d2 -1)}.$$
On the complementary set we estimate
$$\|\cW_{1,\delta}^j g_\rho\|_{L^1(|x|\ge 2 \rho)} \lc
\sum_{k\in \bbZ}
\|\cW_{1,\delta,k}^j g_\rho\|_{L^1(|x|\ge 2 \rho)}\,.$$
From our previous bound for $\cW^j_{q,\delta,k}$ and
the cancellation of the atom,
$$
\|\cW_{1,\delta,k}^j g_\rho\|_1
 \lc 2^{j(d-1)}  \|\widetilde P_{k-j}g_\rho\|_1
\lc 2^{j(d-1)} \min \{1, 2^{k-j}\rho\}.
$$
Now if $\Psi\in \cS^{\theta,j}_{k-j}$ then by
\eqref{Psiestim} with $n=k-j$,
\eqref{cKconv},  and $2^{2j}\delta\le 1$,
$$
\cK^{\theta,j}_{2^{k+j}\delta}*
|\Psi| (x) \lc
\cK^{\theta,j}_{2^{k+j}\delta}*\cK^{\theta,j}_{2^{k-j}}(x) \lc
\frac{(2^k\delta)^d 2^{j(d-1)}}
{(1+ 2^k\delta|\inn{x}{\theta}| +2^j 2^k\delta
|x-\inn{x}{\theta}\theta|)^{d+1}}.
$$
Thus a favorable estimate holds for $2^{k}\rho\ge \delta^{-1}$, namely
$$
\|\cW_{1,\delta,k}^j g_\rho\|_{L^1(|x|\ge 2 \rho)}
\lc 2^{j(d-1)} ( 2^{k}\delta\rho)^{-1}.
$$
These estimates can be summed for $2^k\le \rho^{-1}$ and
$2^k\ge \delta^{-1}\rho^{-1}$.
For the intermediate terms we use the $L^1$ bounds and the Schwarz
inequality
  \begin{align*}
\sum_{\rho^{-1}\le 2^k\le (\rho\delta)^{-1}}
\|\cW_{1,\delta,k}^j g_\rho\|_1
 &\lc 2^{j(d-1)}\sum_{\rho^{-1}\le 2^k\le (\rho\delta)^{-1}}
  \|\widetilde P_{k-j}g_\rho\|_1
\\
&\lc 2^{j(d-1)}[\log (\tfrac{1}{\delta})]^{1/2}
 \Big(\sum_k \|\widetilde P_{k-j}g_\rho\|_1^2\Big)^{1/2}\, ,
\end{align*}
and by Minkowski's integral inequality
and the Littlewood-Paley
inequality  on $H^1$ (see {\it e.g.} \cite{gcrdf})
we have
\Be\label{LPHardy}
 \Big(\sum_k \|\widetilde P_{k-j}g_\rho\|_1^2\Big)^{1/2}\le
\Big\| \Big(\sum_k |\widetilde P_{k-j}g_\rho|^2\Big)^{1/2}\Big\|_1
\lc \|g_\rho\|_{H^1} \lc 1.
\Ee
Now  collect the estimates and  \eqref{atomest} is proved.

\subsubsection*{The weak type estimate}
We observe the pointwise estimate
$$
\sup_k \cK^{\theta',j}_{2^{k+j}\delta}
*\big|\cM^{\theta',j}_{k-j}
P_{k-j} g\big|^q
\lc M^{\theta,j} \big[ \sup_k |\cM^{\theta,j}_{k-j}
P_{k-j} g|^q \big], \quad |\theta-\theta'|\le 2^{-j},
$$
where $M^{\theta,j}$ is the maximal operator associated with 
 tubes of eccentricity $2^{-j}$, with the long side pointing in the direction $\theta$.
For each $(\theta,j)$, $M^{\theta,j}$ is a rescaled version of the Hardy--Littlewood maximal function and therefore satisfies the standard weak-type
$(1,1)$ inequality with a bound independent of  $\theta$ and $j$.

We use the continuous embedding $\ell^q(L^{q,\infty})\subset L^{q,\infty}(\ell^\infty)$ (see \eg  Lemma 2.1 in~\cite{gs}) and dominate
\begin{align}\Big\| \sup_{0<\delta<1/2} |\cW_{q,\delta}^j g|\Big\|_{L^{q,\infty}}
&\lc \Big(\sum_{\theta\in \Theta_j} \Big\| \big(
M^{\theta,j} \big[ \sup_k |\cM^{\theta,j}_{k-j}
P_{k-j} g|^q \big]\big)^{1/q}\Big\|_{L^{q,\infty}}^q\Big)^{1/q}
\notag
\\
&\lc\,\Big(\sum_{\theta\in \Theta_j} \big\|
\sup_k |\cM^{\theta,j}_{k-j}
P_{k-j} g|\, \big\|_{L^q}^q\Big)^{1/q}\, ,
\label{Lqexpression}
\end{align}
and for the last inequality  we have used
 the uniform weak-type $(1,1)$ bounds for the
 $M^{\theta,j}$ together with the identity
$\|F^q\|_{L^{1,\infty}} = \|F\|_{L^{q,\infty}}^q$ for the usual quasinorms.

We may dominate  \eqref{Lqexpression} by
$$
\Big(\sum_{\theta\in \Theta_j}\sum_k \big\|
\cM^{\theta,j}_{k-j}
P_{k-j} g\big\|_{q}^q\Big)^{1/q}
$$
which for $q\ge 2$ has already been estimated by $2^{j(d-2)/q}\|g\|_q$.
Thus the asserted estimate follows in this range.

For $1<q<2$ we claim that
$$\Big(\sum_{\theta\in \Theta_j} \big\|
\sup_k |\cM^{\theta,j}_{k-j}
P_{k-j} g|\, \big\|_{q}^q\Big)^{1/q} \lc 2^{j(\frac dq -1)} (1+j)^{1/q-1/2} \|g\|_q.
$$
This follows by complex  interpolation from the estimate for
$q=2$ already proved above and an
$H^1\to \ell^1(L^1(\ell^\infty))$ bound. Again it suffices to consider
 an $L^2$-atom $g_\rho$ supported on a ball  of radius $\rho$
centered at the origin and we need to check
$$\sum_{\theta\in \Theta_j} \big\|
\sup_k |\cM^{\theta,j}_{k-j}
P_{k-j} g_\rho
|\, \big\|_1 \lc 2^{j(d -1)} (1+j)^{1/2}.
$$
Now on the ball of radius $2\rho$ we use an $L^2$ estimate and the trivial estimation $|\cM^{\theta,j}_{k-j} f|
\lc M^{\theta,j}(f)$ to obtain
\begin{align*}
&\sum_{\theta\in \Theta_j} \big\|
\sup_k |\cM^{\theta,j}_{k-j}
P_{k-j} g_\rho|\, \big\|_{L^1(|x|\le 2\rho)}
\,\lc \,\rho^{d/2} \sum_{\theta\in \Theta_j} \big\|
\sup_k |\cM^{\theta,j}_{k-j}
P_{k-j} g_\rho|\, \big\|_{2}
\\
&\lc
\rho^{d/2} \sum_{\theta\in \Theta_j} \Big\|
\Big(\sum_k (M^{\theta,j}(
P_{k-j} g_\rho))^2\Big)^{1/2} \Big\|_{2}
\\
&\lc
\rho^{d/2} \sum_{\theta\in \Theta_j} \Big(\sum_k \|P_{k-j} g_\rho\|_{2}^2\Big)^{1/2}
\lc 2^{j(d-1)}
\rho^{d/2}  \|g_\rho\|_2 \,\lc \,2^{j(d-1)}\,.
\end{align*}
For $|x|\ge 2\rho$ we replace the sup in $k$ by the sum.
By standard $L^1$ estimates and using the cancellation of the atom
 we have
$$ \big\|
 \cM^{\theta,j}_{k-j}
P_{k-j} g_\rho\big\|_1 \lc
\|P_{k-j} g_\rho\big\|_1 \lc
\min \{2^{k-j}\rho, 1\}
$$
and, using estimates of  the kernels,
$$ \big\|  \cM^{\theta,j}_{k-j}
P_{k-j} g_\rho\big\|_{L^1(|x|\ge 2\rho)} \lc (2^{k-2j}\rho)^{-1}\quad \text{ if }\ 2^k\rho>2^j
\,.
$$
Thus $\sum_{\theta\in \Theta_j}\sum_{2^k\rho \notin [2^j, 2^{2j}]}
\big\|
 \cM^{\theta,j}_{k-j}
P_{k-j} g_\rho \big\|_{L^1(|x|\ge 2\rho)}
\lc 2^{j(d-1)}.$
Moreover, for the intermediate terms,
$$\sum_{\theta\in \Theta_j}\sum_{2^k\rho \in
[ 2^j, 2^{2j}] }
\big\|
 \cM^{\theta,j}_{k-j}
P_{k-j} g_\rho\, \big\|_{1}
\lc 2^{j(d-1)}\sum_{2^k\rho \in  [ 2^j, 2^{2j}]}\|P_{k-j} g_\rho\|_{1}
\lc 2^{j(d-1)}(1+j)^{1/2},$$
by the argument in \eqref{LPHardy}.
 We combine these estimates  and the $L^1$ bound is proved.
\end{proof}

\section{Multipliers and the bilinear adjoint restriction theorem}
\label{biladjrestr}

In this section we prove bilinear estimates for multiplier transformations,
under suitable separation conditions. The proofs rely on

\medskip

\noi{\bf Tao's bilinear adjoint restriction theorem
(\cite{ta},~\cite{le-tr}).} {\it Let $b>1/2$ and $p>2+4/d$.
There exist $\eps_{\!\circ}>0$,
 $N_\circ\in \bbN$ and $C$,
 depending on $b$, $p$ and  $d$,
that for all functions $h$ defined on $[-b,b]^{d-1}$ and satisfying
\begin{equation}\sup_{\omega\in [-b,b]^{d-1}}
\max_{\substack{|\alpha|\le N_\circ\\i=1,2}}
 |\partial^\alpha_\omega h(\omega)|\le \eps_{\!\circ}
\end{equation}
the following holds:
For all pairs of functions $(F_1,F_2)$  with
$\dist (\supp(F_1), \supp (F_2))\ge 1/2$  and  $F_i\in L^2([-b,b]^{d-1})$,}
\Be\label{bilextension} \Big(\int \Big| \prod_{i=1,2}
\int_{[-b,b]^{d-1}} F_i(\omega)
\exp\big(\im \inn{x'}{\omega} + \im x_d(|\omega|^2/2+h(\omega))\big)
d\omega\Big|^{p/2}dx\Big)^{2/p} \lc  \prod_{i=1,2}
\big\|F_i\big\|_{L^2}. \Ee

We will need to  consider  families of hypersurfaces which depend
on a parameter $s$ and which,   for fixed $s$, are  small perturbations of the paraboloid $\xi_d=|\xi'|^2/2$, where $\xi'=(\xi_1,\dots, \xi_{d-1})$.
These lead to ``elliptic'' phase-functions as considered in~\cite{tavave},~\cite{ta}.

\medskip

\noi{\bf Definition.} We denote by $\Ellip(b,\eps,N_\circ)$
the class of functions $(\xi',s)\mapsto \gamma(\xi',s)$ defined on
$[-b,b]^{d-1}\times (-1,1)$ which are of the form
$$\gamma(\xi',s)= \frac{|\xi'|^2}{2} -s +h(\xi',s)
,$$
with
\begin{equation}\sup_{\substack{\omega\in [-b,b]^{d-1}\\ s\in (-1,1)}}
\max_{\substack{|\alpha|\le N_\circ\\i=1,2}}
 |\partial^\alpha_{\omega,s} h(\omega,s)|\le \eps.
\end{equation}
We may and shall assume in what follows that $N_\circ$ is large, say
$N_\circ>10d$.

We now consider  Fourier multipliers depending on a parameter $s\in (-1,1)$,
supported in a tubular neighborhood of $(\xi',\gamma(\xi',s))$.

\begin{lemma} \label{LpL2} Let
$p> 2+ 4/d$, $b>1/2$. There are $\eps$, $N_\circ$, depending on $d,$ $b$, and $p$, so that the following holds for all $\delta_\circ<1/2$.

 Let $\gamma \in  \Ellip(b,\eps,N_\circ)$.
For $|s|\le 1$ and $i=1,2$, let $a_i(\cdot,s)$ be multipliers,
supported on $[-b,b]^{d-1}$, satisfying the conditions
\begin{subequations}\label{ai-assumpt}
\Be \label{bdassumpt}
|a_i(\xi,s)|\le 1\,,
\Ee
\Be \label{aisupport}
a_i(\xi,s)=0\quad \text{ if }\ |\xi_d-\gamma(\xi',s)|\ge \delta_\circ,
\Ee
and
\Be\label{aiseparate-assumpt}
(\xi',\xi_d)\in \supp a_1(\,\cdot\,,s),\quad
(\widetilde\xi',\widetilde \xi_d)\in \supp a_2(\,\cdot\,,s) \,\implies\,
|\xi'-\widetilde \xi'|\ge 1\,.
\Ee
\end{subequations}
Then,  for all pairs of $L^2$ functions $(f_1,f_2)$,
\Be \label{LpL2bilmult}
\Big\| \int_{-1/2}^{1/2} \prod_{i=1,2} a_i(D,s) f_i \,ds\Big\|_{p/2} \lc \delta_\circ^2 \prod_{i=1,2}\|f_i\|_2\,.
\Ee
\end{lemma}

\begin{proof} For fixed $s$ we introduce coordinates
 $$\xi=\Gamma^s(\xi',\tau):=(\xi', \gamma(\xi',s)+\tau)$$
in the Fourier integral.
We then need to  estimate the $L^{p/2}$ norm of
 %
 %
$$
\int_{-1/2}^{1/2} \prod_{i=1,2} \biggl[ \int_{-\delta_\circ}^{\delta_\circ}
\int_{[-b,b]^{d-1}}
[a_i\widehat f_i](\Gamma^s({\xi^{i}}',\tau^i))
e^{\im(\inn{x'}{{\xi^i}'}+ x_d( \gamma({{\xi^{i}}'},s)+\tau^i))} d{\xi^{i}}'
d\tau^i\biggr] ds;
$$
here we denote by $({\xi^{i}}',\tau^i)$
the variables  in the two different copies of $\bbR^d$.
By  Minkowski's integral inequality the $L^{p/2}$ norm is dominated by
$$
\iiint\limits_{[-\delta_\circ,\delta_\circ]^2\times [-\frac 12,\frac 12]}
 \Big(\int \Big| \prod_{i=1,2}\Big[\int\limits_{[-b,b]^{d-1}}
 [a_i
\widehat  f_i](\Gamma^s({\xi^{i}}',\tau^i))
e^{\im(\inn{x'}{{\xi^i}'}+x_d \gamma(\xi^{i},s))}
d\xi^i\Big]
\Big|^{p/2} dx\Big)^{2/p}ds \,d\tau^1d\tau^2.
$$
By the bilinear adjoint restriction theorem and the boundedness of $a_i$
this is estimated by
$$\iiint\limits_{[-\delta_\circ,\delta_\circ]^2\times [-\frac 12,\frac 12]}
\prod_{i=1,2}
\Big(\int_{[-b,b]^{d-1}}\big| \widehat f_i(\Gamma^s({\xi^{i}}',\tau^i))
\big|^2 d{\xi^i}'\Big)^{1/2} ds\,d\tau^1 d\tau^2.
$$
We apply the Schwarz inequality in the $s$ variable. Then
 for
fixed $\tau^1$, $\tau^2$, we
 change
 variables $(\xi',s)\mapsto \xi= (\xi',  \gamma(\xi',s)+\tau^i)$, using that
 $\partial_s \gamma(\xi',s)= 1+O(\eps)$.
Thus the last displayed expression is estimated by
\begin{multline*}
 \iint\limits_{[-\delta_\circ,\delta_\circ]^2}
\prod_{i=1,2}
\Big(\int\limits_{[-b,b]^{d-1}\times[-1,1]} \big |\widehat f({\xi^i}',\gamma({\xi^i}',s)+\tau^i)\big|^2 ds\,
d {\xi^i}'\Big)^{1/2} d\tau^1d\tau^2
\\
\lc   \iint\limits_{[-\delta_\circ,\delta_\circ]^2}
 \prod_{i=1,2}\|\widehat f_i\|_2\,d\tau^1d\tau^2
\lc \, \delta_\circ^2 \prod_{i=1,2}\| f_i\|_2\,.
\end{multline*}
\end{proof}

In what follows
we will use the notation
$\prodstartext[\cE_i] = \cE_1\overline{\cE_2}$
for products involving a
complex conjugate.

\begin{proposition} \label{bilweight} Let $q\in[1,\frac{d+2}{2})$ and  $b>1/2$. There are
$\eps$,  $N_\circ$, depending on $d,$ $b$ and $q$, so that the following holds for all $\delta_\circ<1/2$.

Let $\gamma \in  \Ellip(b,\eps,N_\circ)$. For $|s|\le 1$ and
$i=1,2$, let $a_i(\cdot,s)$ be multipliers, supported on
$[-b,b]^{d-1}$, satisfying the conditions
\begin{subequations}\label{mi-assumpt}
\Be \label{diffassumpt}
|\partial_{\xi}^\alpha a_i(\xi,s)|\le \delta_\circ^{-|\alpha|}, \quad |\alpha|\le
d+2,
\Ee
\Be\label{ai-support-assumption}
a_i(\xi,s)=0 \quad \text{ if }\ |\xi_d-\gamma(\xi',s)|\ge \delta_\circ,
\Ee
and
\Be\label{separate-assumpt}
(\xi',\xi_d)\in \supp a_1(\,\cdot\,,s),\quad
(\widetilde\xi',\widetilde \xi_d)\in \supp a_2(\,\cdot\,,s) \,\implies\,
|\xi'-\widetilde \xi'|\ge 1\,.
\Ee
\end{subequations}
Then
\begin{multline}\label{bilinearweight}
\Big|\int_{\bbR^d} \int_{-1/2}^{1/2}
\prodstar[a_i(D,s)f_i(x)]
w(x) \,dx\,ds \Big|
\\ \le C_N \delta_\circ^{2}
\prod_{i=1,2} \Big(\int|f_i(x)|^2 \Big(\int \frac{|w(x-y)|^q}
{(1+\delta_\circ|y|)^{(d+1)q}}dy\Big)^{1/q} dx\Big)^{1/2}\,.
\end{multline}

\end{proposition}

\begin{proof}

We dyadically decompose  the kernel of the convolution operators. Let
$\rho_{0}$ be a  $C^\infty_c(\bbR)$ function supported on $(-1,1)$ and equal to one on $[-1/2,1/2]$ and define, for $x\in \bbR^d$,
$$\Phi_0(x)= \rho_0(\delta_\circ| x|), \qquad \Phi_j(x)=
\rho_0(2^{-j}\delta_\circ|x|)-\rho_0(2^{1-j}\delta_\circ|x|),\quad j\ge 1.$$
Then the $\{\Phi_j\}_{j=0}^\infty$ form a radial partition of unity.
We thus need to bound the sum
$$\sum_{j_1, j_2\ge 0}
\Big|\int_{\bbR^d} \int_{-1/2}^{1/2} \prodstar\Big[
\int a_i(D-\eta^i,s)f_i(x) \widehat \Phi_{j_i}(\eta^i)\,d\eta^i
\Big]
w(x)\,dx\,ds \Big|.
$$
By symmetry considerations it suffices to consider  the terms with
$0\le j_1\le j_2$.
The desired estimate then follows if we can show that
\begin{multline}\label{fixedjweightedbd}
\Big|\int_{\bbR^d} \int_{-1/2}^{1/2}
 \prodstar\Big[
\int a_i(D-\eta^i,s)f_i(x) \widehat \Phi_{j_i}(\eta^i)\,d\eta^i
\Big]ds\,
w(x)\,dx \Big|
\lc \\ C(L)
\delta_\circ^2 2^{-j_2 L} \prod_{i=1,2} \Big(\int|f_i(x)|^2 \Big(\int_{|y|\le C
2^{j_2}\delta_\circ^{-1}} |w(x+y)|^q dy\Big)^{1/q}dx\Big)^{1/2}
\end{multline}
for
$0\le j_1\le j_2$, $L\le d+2$. We shall first verify this inequality for $L=0$ and then provide the modification for $0<L\le d+2$.

We now form a grid $\fQ(j_2)$ of dyadic cubes of sidelength $2^{j_2}\delta_\circ^{-1}$.
For every $Q\in \fQ(j_2)$ let $Q^*$ be the double cube with same center as $Q$.
By the support properties of the kernels, and $j_1\le j_2$,  we have
$$
\int a_i(D-\eta^i,s)[f_i\chi_{\bbR^d\setminus Q^*}](x) \widehat
\Phi_{j_i}(\eta_i)\,d\eta_i \,=0\quad \text{ if } x\in Q, \quad
Q\in \fQ(j_2), \quad i=1,2.$$ Thus the left hand side of
\eqref{fixedjweightedbd} is equal to \Be\label{Qlocalizationdone}
\Big|\sum_{Q\in \cQ(j_2)}\int_{Q} \int_{-1/2}^{1/2}
\prodstar\Big[\int a_i(D-\eta^i,s)[\chi_{Q^*}f_i](x) \widehat
\Phi_{j_i} (\eta^i)\,d\eta^i\Big] ds\,w(x)\,dx \Big|\,. \Ee We use
the formula $m(D-\eta)f= \Mod_\eta m(D) [\Mod_{-\eta }f]$ where
$\Mod_\eta g(x) = g(x)\,e^{\im\inn{x}{\eta}}.$ In order to obtain
\eqref{fixedjweightedbd} for $L=0$ (which is efficient for $j_2=0$)
we use H\"older's inequality to estimate \eqref{Qlocalizationdone}
by \Be \label{nachHoelder} \iint |\widehat \Phi_{j_1}(\eta^1)
\widehat \Phi_{j_2}(\eta^2)| \sum_{Q\in \fQ(j_2)} \Big\|
\int_{-1/2}^{1/2} \prodstar\Big[
a_i(D,s)\Mod_{-\eta^i}[\chi_{Q^*}f_i]\Big]\Big\|_{q'}
\|w\chi_Q\|\ci{q} d\eta^1d\eta^2\,. \Ee Now
\begin{align*}
&\Big\| \int_{-1/2}^{1/2}
\prodstar\Big[
a_i(D,s)\Mod_{-\eta^i}[\chi_{Q^*}f_i]\Big]\,ds\Big\|_{q'}\,
\\
\le\ & \int_{-1/2}^{1/2}
\Big\|
\prodstar\Big[
a_i(D,s)\Mod_{-\eta^i}[\chi_{Q^*}f_i]\Big]\Big\|_{q'}\,ds
\\=\ & \int_{-1/2}^{1/2}
\Big\|
\prod_{i=1,2}
a_i(D,s)\Mod_{-\eta^i}[\chi_{Q^*}f_i]\Big\|_{q'}\,ds\,.
\end{align*}
By our assumption on $q$ we have $q'=p/2$ for some  $p>2+4/d$ and therefore
we can
use Lemma ~\ref{LpL2} to
bound  the last displayed expression by
$$ \delta_\circ^2
\prod_{i=1,2}\|f_i\chi_{Q^*}\|_2\,.$$

Since the $L^1$ norms of $\widehat \Phi_{j_1}$ are $O(1)$, uniformly in $\delta_\circ$,
we have
\begin{align*}
\text{\eqref{Qlocalizationdone}}\,&\lc\,\delta_\circ^2 \sum_{Q\in \fQ(j_2)}\|f_1\chi_{Q^*}\|_2 \|f_2\chi_{Q^*}\|_2 \|w\chi_Q\|_q
\\&\lc \delta_\circ^2 \prod_{i=1,2}
\Big(\sum_{Q\in \fQ(j_2)}\|f_i\chi_{Q^*}\|_2^2
 \|w\chi_Q\|_q\Big)^{1/2}
\\ &\lc \delta_\circ^2 \prod_{i=1,2}
\Big(\int_{Q^*}|f_i(x)|^2\Big(\int_{|y|\le C2^{j_2}\delta_\circ^{-1}}|w(x+y)|^q dy\Big)^{1/q}dx\Big)^{1/2}
\end{align*}
which yields
\eqref{fixedjweightedbd} for $L=0$.

We now turn to the case $L\le d+2$  where we need to improve the above estimate
by a factor
of $C(L) 2^{-j_2L}$.
We expand the convolution $\int
a_2(\xi-\eta,s)\widehat\Phi_{j_2}(\eta)\, d\eta$ by a Taylor
expansion about $\eta=0$. Since $\Phi_{j_2}$ vanishes in a
neighborhood of the origin the integrals $\int \widehat \Phi_{j_2}
(\eta) P(\eta)\, d\eta$ are zero for any polynomial $P$. Thus only
the integral remainder term in the Taylor expansion survives and we
obtain
\begin{align*}a_2(\,\cdot\,,s)*\widehat{\Phi}_{j_2}(\xi) =
\int_{0}^1 \frac{(1-\sigma)^{L-1}}{(L-1)!} \int \widehat{\Phi}_{j_2}
(\eta)\inn{-\eta}{\nabla_\xi}^L a_2(\xi-\sigma\eta,s)\, d\eta
\,d\sigma\,.
\end{align*}
We repeat the above argument in which we now  have
to bound
\begin{multline*}
\int_{\eta^1}\int_{\eta^2}\int_{0}^1
\big|\widehat \Phi_{j_1}(\eta^1)\,d\eta^1\big|\,
\big|\widehat \Phi_{j_2}(\eta^2)\,d\eta^2\big|\,
\sum_{Q\in \cQ(j_2)}\int_{x\in Q} |w(x)| \quad \times \\
\Big| \int_{-1/2}^{1/2}
 a_1(D-\eta^1,s)[\chi_{Q^*}f_1](x) \overline{ \inn{-\eta^2}{\nabla}^L
a_2(D-\sigma\eta^2,s)[\chi_{Q^*}f_2](x)} \,ds\Big| \,
dx \, d\sigma\,d\eta^1d\eta^2
\end{multline*}
in place of \eqref{Qlocalizationdone}.
In the estimate we may replace $\inn{-\eta^2}{\nabla}^L$  with
$ (\eta^2)^\alpha \partial_\xi^\alpha$, for any multiindex $\alpha$
with $|\alpha|=L$.
As above we continue with H\"older's inequality, and this time
Lemma~\ref{LpL2}  and the differentiability assumptions on $a_{2}$
yield for $|\alpha|=L$
\begin{multline*}
\int_{0}^1\Big\| \int_{-1/2}^{1/2}
a_1(D,s)[\chi_{Q^*}\Mod_{-\eta^1}f_1]
\overline{\partial_\xi^\alpha a_2(D,s)\Mod_{-\sigma\eta^2}[\chi_{Q^*}f_2]}\,ds
\Big\|_{q'}d\sigma
 \\ \lc\, \delta_\circ^{2 -L}
\prod_{i=1,2}\|f_i\chi_{Q^*}\|_2\,.
\end{multline*}
The  loss of $\delta_\circ^{-L}$  in the previous formula is
(more than)  mitigated by
$$
\iint |\widehat \Phi_{j_1}(\eta^1) (\eta^2)^{\alpha}\widehat {\Phi}_{j_2}(\eta^2)|
\,d\eta^1\,d\eta^2 \lc 2^{-j_2L} \delta_\circ^{L},\qquad |\alpha|=L.
$$
Thus  the above argument yields  \eqref{fixedjweightedbd} for also for $0<L\le d+2$.
\end{proof}

\section{Proof of the weighted inequality} \label{proofweighted}
In  this section we prove  inequality  \eqref{weightboundSdelta}
of Theorem~\ref{weightL2thm}.
We mainly focus on a local inequality (with $t$-interval $[1,2]$) which for later application  we formulate for slightly more general multipliers.
Instead of $S^\delta_t$ we consider  operators $\cS^\delta_t$ defined by
$$
\widehat {\cS^{\delta}_t f}(\xi)=
\phi\big(\delta^{-1}(1-\tfrac{|\xi|^2}{t^2}) \beta(\xi,t)\big)\widehat f(\xi)
$$
where   $\phi$ is as in \eqref{schwartzestimates}, $\beta$ is a nonvanishing $C^\infty$ function on the set of  $(\xi,t)$ with $1/2< t< 5/2$ and
 $1/2\le |\xi|\le 4$.
Of course $\beta(\xi,t)=1$ in Theorem~\ref{weightL2thm}.


\begin{theorem}\label{Tdeltat}
Let $d\ge2$ and $q\in[1,\frac{d+2}{2})$.
Then, for $0<\delta<1/2$,
\Be\label{unitscale}\begin{aligned}
&\int_{\R^d}\int_{1}^2|\cS^\delta_tf(x)|^2\frac{dt}{t}\, w(x)\,dx
\lc \delta^{2-d/q}\int_{\bbR^d} |f(x)|^2\,Ww(x) \,dx
\end{aligned}
\Ee
with
$$
Ww(x)=
\sum_{1\le 2^{2j}<
\delta^{-1}} 2^{-2j(\frac dq -1)} M\!\circ\!\cW^j_{q,\delta,0} w (x) +\delta^{\frac dq -1}  M\! \circ \fM_{\sqrt\delta}w(x).
$$
\end{theorem}

We now show that Theorem \ref{Tdeltat}   (with $\beta\equiv1$)
implies
assertion \eqref{weightboundSdelta} of Theorem~\ref{weightL2thm}.
Let $W_kw(x)= W[w(2^{-k}\,\cdot\,)](2^kx)$. By \eqref{unitscale} and rescaling we see that
\begin{align*}
&\int_{\R^d}\int_{0}^\infty|S^\delta_tf(x)|^2\frac{dt}{t}\, w(x)\,dx
\\
=\ & \sum_{k\in \bbZ}
\int_{\R^d}\int_{1}^2\big|S^\delta_{2^k s}\sum_{i=-2}^2 P_{k+i}f(x)\big|^2\frac{ds}{s}\, w(x)\,dx
\\
\lc\ &\delta^{2-d/q}
\sum_{k\in \bbZ}\sum_{i=-2}^2
\int_{\R^d}\int_{1}^2\big|
P_{k+i}f(x)\big|^2 \,W_k w(x)\,dx
\frac{ds}{s}
\\
\lc\ &\delta^{2-d/q}
\int_{\R^d}\int_{1}^2|f(x)|^2
M(\sup_k|W_k w|^s)^{1/s}(x)\,dx
\end{align*}
and the last inequality is a consequence of Coifman's improvement
of the C\'ordoba--Fefferman
weighted norm inequality for singular integrals (see for example \cite[p. 417]{gcrdf}).
Now $M(\sup_k|W_k w|^s)^{1/s}(x)\lc \fW_{q,\delta}w$, by Minkowski's inequality
(\cf. \eqref{fWqdel}).

Theorem~\ref{Tdeltat}  implies the following sharp $L^p$ results,
by  a duality argument, the boundedness results of Proposition
\ref{maximaltheorem} and the Marcinkiewicz interpolation theorem.
 
\begin{corollary} \label{Tdeltatcor}
Let $d\ge 2$ and $p\in(\frac{2(d+2)}{d},\infty)$. Then, for $0<\delta<1/2$,
$$
\Big\|\Big(\int_{1}^2|\cS^\delta_tf(x)|^2\frac{dt}{t}\Big)^{1/2}\Big\|_p \lc
\delta^{1-d(\frac 12-\frac 1p)} \|f\|_p.
$$
Moreover if $\beta\equiv 1$ and $S^\delta_t$ is as in
\eqref{Sdeltat}
then
$$
\Big\| \Big(\int_{0}^\infty|S^\delta_tf(x)|^2\frac{dt}{t}\Big)^{1/2}\Big\|_p\lc \,
\delta^{1-d(\frac 12-\frac 1p)} \|f\|_p.
$$
\end{corollary}
The remainder of this section is devoted to the proof of
Theorem~\ref{Tdeltat}.

\subsection*{Preliminary considerations.}
We  begin with a rescaled variant of Proposition~\ref{bilweight}.
Such rescaling arguments
have been used in \cite{tavave}, \cite{ta}, \cite{le} and elsewhere.
In what follows fix a function
$\zeta\in C^\infty(\bbR^{d})$ supported in $\{y: |y|\le 1/8\}$ and define
a convolution operator with homogeneous multiplier by
\Be\label{Qjth}\widehat {Q^{\theta,j} f} =\zeta\big(2^j\big(\tfrac {\xi}{|\xi|}-\theta\big)
\Ee

In order to reduce estimates  to
Proposition \ref{bilweight} by rescaling
 we will need  to localize all multipliers
to a narrow sector $\{\xi:
|\tfrac{\xi}{|\xi|}-u|\le \eps_1 \}$
where $u$ is a unit vector
and $\eps_1$ is a small constant.
\begin{lemma} \label{scaledbil}
Given $C>1$
 there are small $\eps_1,\eps_2\in (0,1/8)$  depending on $q$, $d$ and $C$ and the function $\beta$
so that the following statement holds for $C2^{-j}<\eps_1$ and
$2^{2j}\delta<\eps_2$.

Let   $\psi$  be
supported in a ball of radius $\eps_1$ contained in $\{\xi:1/4\le|\xi|\le 4\}$
such that
 $\|\partial^\alpha \psi\|_\infty\le 1$
for $|\alpha|\le d+2$, let $\theta_1, \theta_2 \in  S^{d-1}$ be
 such that $2^{-j-1}\le |\theta_1-\theta_2|\le
C 2^{-j}$, let
$t_0\in[1,2]$ and $J$ be an interval of length $2^{-2j-2}$ containing $t_0$.
 Then
\begin{multline*}
\Big|\int\int_{J}
\cS_t^\delta  Q^{\theta_1,j} \psi (D)f_1(x)
\overline
{\cS_t^\delta  Q^{\theta_2,j} \psi (D)f_2(x)}\, \frac{dt}{t}\, w(x)\,dx \Big|
\\
\lc \delta (2^{2j}\delta)^{1- d/q}
\prod_{i=1,2} \Big( \int |f_i(x)|^2 \big(\cK^{\theta_i,j}_{2^j\delta} *|w|^q(x)\big)^{1/q} dx\Big)^{1/2}\,.
\end{multline*}
\end{lemma}

\begin{proof}
Set $\vth=\frac{\theta_1+\theta_2}{|\theta_1+\theta_2|}$ and
$e_d=(0,\dots,0,1)$. Take $R_\vth$ to be a rotation satisfying
$R_\vth e_d=\vth$ and let it act on functions by $\cR_\vth
f(y)=f(R_\vth y)$. Let $ \cA_j f(y)= f(A_jy)$, where
$A_jy=(2^{j}y_1,\dots, 2^j y_{d-1},- 2^{2j}y_d)$. Finally we set
$\cD_{t_0} f(y)= f(t_0y)$ and, as before, the modulation $\Mod_{a}$
is defined by $\Mod_{a} g(x) = g(x)\,e^{\im \inn{x}{a}}.$

For fixed $\delta, j$ and $i=1,2,$
the multiplier for
$S_t^\delta  Q^{\theta_i,j}\psi (D)$ is given by
$$m_{\theta_i}(\xi,t) =\psi (\xi) \phi \Big(\delta^{-1}(1-\tfrac {|\xi|^2}{t^2})
\beta(\xi,t)\Big) \zeta\Big(2^j (\tfrac{\xi}{|\xi|}-\theta_i)\Big).$$
Let $\Xi_\vth(\eta)=t_0R_\vth(e_d+A_{-j}\eta)$ and $t(s)= t_0(1+2^{-2j}s)$.
A rescaled multiplier depending on the parameter
$s= 2^{2j}(t_0^{-1}t-1)\in [-1/2,1/2]$ is defined by
%
%
\begin{align*}
M_{\theta_i}(\eta, s)&:=
m_{\theta_i}(\Xi_\vth(\eta),  t(s))
\\&= \phi \Big(\delta^{-1}\big(1- \tfrac{|e_d+A_{-j}\eta|^2}{(1+2^{-2j}s)^2}\big)
 \beta\big(\Xi_\vth(\eta), t(s)\big)\Big)\,
\zeta\big( 2^jR_\vth\big[
\tfrac{e_d+A_{-j}\eta}{|e_d+A_{-j}\eta|}
- R^{-1}_\vth\theta_i\big]\big)\psi \big(\Xi_\vth(\eta)\big).
\end{align*}
Now compute
\begin{align}
1- \frac{|e_d+A_{-j}\eta|^2}{(1+2^{-2j}s)^2}
&=2^{1-2j} \Big( \frac{s(1+2^{-2j-1}s)+\eta_d-|\eta'|^2/2-\eta_d^2 2^{-2j-1}}{(1+2^{-2j}s)^2}\Big)
\notag
\\&= 2^{1-2j}
\Big(s+\eta_d-\frac{|\eta'|^2}{2}\Big)+ 2^{-4j}
r_j(\eta,s),
\label{paraboloid}
\end{align}
where  $r_{j}$ is a quadratic polynomial in $\eta$ with coefficients uniformly
bounded in $s$, $j$.
  Moreover the supports of the functions
$\zeta\big( 2^jR_\vth\big[
\tfrac{e_d+A_{-j}\eta}{|e_d+A_{-j}\eta|}- R^{-1}_\vth\theta_i\big]\big)$,
$i=1,2$, are uniformly separated  and these functions have derivatives
with bounds uniform in $j$.

Now let $b=10d C$, and set
$M= 1+\sum_{|\alpha|\le 2}\max_{|\xi|\le b}
\max_{|s\le 1}|\partial_\eta^\alpha r_j(\eta,s)|$,
for the error terms $r_j$ in
\eqref{paraboloid}. Let
$\eps$ be as   in Proposition \ref{bilweight},
 and choose $\eps_1$ small compared to $(2d)^{-1} (\eps/M)^{1/2}$.
By the assumed separation property we have
$$\cS_t^\delta  Q^{\theta_1,j} \psi (D)f_1\,
\overline{\cS_t^\delta  Q^{\theta_2,j} \psi(D)f_2}\equiv 0,\quad 2^{-j}\ge d\eps_1\,.$$
For the relevant complementary range we have $2^{-2j}M <\eps$ so that
the functions
$(\eta',s)\mapsto |\eta'|^2/2 -s +2^{-2j-1} r_j(\eta',s)$ belong to
$\Ellip(b,\eps, N_\circ)$.
Since $\beta$ is smooth and satisfies an inequality $2\eps_2\le |\beta(\xi,t)|\le (2\eps_2)^{-1}$ for $1/4\le |\xi|\le 4$,
the formula \eqref{paraboloid} allows us
to apply Proposition~\ref{bilweight} with $a_i(D,s)= M_{\theta_i}(D,s)
(1+2^{-2j} s)^{-1/2}$ and
 $\delta_\circ= (2\eps_2)^{-1} \delta 2^{2j}<1/2$.

We obtain
\begin{multline*}
\Big|\int_{\bbR^d} \int_{-1/2}^{1/2}
\prodstar[M_{\theta_i}(D, s) f_i(x)]
 w(x)\, dx\,\frac{ds}{1+2^{-2j}s} \Big|
\\ \lc (2^{2j}\delta)^{2}
\prod_{i=1,2} \Big(\int|f_i(x)|^2 (2^{2j}\delta)^{-d/q}\cV_{2^{2j}\delta} w(x)\,dx\Big)^{1/2}\,,
\end{multline*}
where $$\cV_{\delta_\circ} w(x)=\Big(\int \frac{\delta_\circ^d|w(x-y)|^q}{(1+\delta_\circ|y|)^{(d+1)q}}dy
\Big)^{1/q} .$$
Now, with $t=t_0(1+2^{-2j }s)$, we have
$m_{\theta_i}(\xi, t)= M_{\theta_i}(A_j(R_\vartheta^{-1}t_0^{-1}\xi-e_d), s),$
so that
$$\cS_t^\delta  Q^{\theta_i,j}\psi(D) f_i(x)
=
\cD_{t_0} \cR_\vth^{-1} \Mod_{e_d}\cA_{-j} M_{\theta_i}(D,s)g_i(x)\ \text{ with }\
g_i=\cA_{j}\Mod_{-e_d} \cR_\vth \cD_{t_0}^{-1}f_i.
$$
This leads to
\begin{align*}
&\Big|\int\int_{I_{j,T} }
\prodstar[\cS_t^\delta  Q^{\theta_i,j}\psi(D) f_i(x)]
\frac{dt}{t}\, w(x)\,dx \Big|
\\=\ & 2^{-2j}
\Big|\int\int_{-1/2}^{1/2}\prodstar[
\cD_{t_0} \cR_\vth^{-1} \Mod_{e_d}\cA_{-j} M_{\theta_i}(D,s)g_i(x)]
\frac{ds}{1+2^{-2j} s}\, w(x)\,dx \Big|
\\
\lc\ & 2^{-2j} (2^{2j} \delta)^2
\prod_{i=1,2} \Big(\int |f_i(x)|^2 (2^{2j}\delta)^{-d/q}
\cD_{t_0}\cR_\vth^{-1} \cA_{j}
[ \cV_{2^{2j}\delta} (\cA_{j}\cR_\vth \cD_{t_0}^{-1} w)](x)\,dx\Big)^{1/2}\\
\lc\ & 2^{2j(1-d/q)} \delta^{2-d/q}
\prod_{i=1,2} \Big(\int |f_i(x)|^2
 \big(\cK^{\vth, j}_{2^jt_0\delta} *|w|^q(x)\big)^{1/q} dx\Big)^{1/2}.
\end{align*}
The assertion now follows from
$\cK^{\vth, j}_{2^jt_0\delta} (x)\approx
\cK^{\theta_i,j}_{2^j\delta} (x)$ since $t_0\in [1,2]$ and  $|\vth-\theta_i|\lc 2^{-j}$ for $i=1,2$.
\end{proof}

A version of the following lemma  is originally due to Carleson (unpublished);
slightly different forms  can be found in ~\cite{co2},
~\cite{ru0} and~\cite{se1}.
For the sake of  completeness we include the proof.

\begin{lemma}\label{carl}
Let $A$ be an invertible linear transformation and $A^t$ its transpose.
 Suppose that $\{m_k\}_{k\in\bbN}$ have disjoint supports.
Then for $s\ge 0$ and almost every $x\in \bbR^d$
$$\Big(\sum_{k}\big|\cF^{-1}[m_k(A^{t}\,\cdot\,)
  \widehat f \,](x)
\big|^2\Big)^{1/2} \le C\sup_k
  \|m_k\|_{L^2_s}
\Big(\int
\frac{\det (A^{-1})}{(1+ |A^{-1}y|^2)^s } |f(x-y)|^2 dy \Big)^{1/2}.
$$
\end{lemma}

\begin{proof}
We can assume that $m_k=0$ for all but
finitely many $k$.
Since
$\cF^{-1}[m_k(A^{t}\,\cdot\,)
  \widehat f \,](x)=
\cF^{-1}[m_k
  \widehat {f(A\,\cdot\,)}](A^{-1}x)$
we can reduce to the case where 
$A$ is the identity transformation.
Also, using an  analytic interpolation argument, we may assume that $s\in \bbN\cup\{0\}$.

Let  $g_k = \cF^{-1}[m_k]$. Then
$\sum_{k}\big|\cF^{-1}[m_k  \widehat f \,](x)\big|^2
= \sup_{\substack {\|a\|_{\ell^2(\bbZ^d)}\le 1}}
|\sum_k a_k  \, g_k *f(x)|^2.$
Now, for each fixed $a\in \ell^2(\bbZ^d)$, we  apply  the Schwarz inequality in the convolution integral and then Plancherel's theorem to obtain
$$
\Big|\sum_k a_k  \,g_k*f (x)\big|
\lc\Big\|\sum_k a_k  \, m_k\Big\|_{L^2_s} \Big(\int
\frac{|f(x-y)|^2}{(1+|y|^{2})^{s}} dy \Big)^{1/2}.
$$
Thus, we are done as for $s\in \bbN\cup\{0\}$,
\begin{equation*}\label{sob}
\Big\|\sum_k a_k  \, m_k\Big\|^2_{L^2_s}\lc\sum_{|\alpha|\le
s}\Big\|\sum_k a_kD^\alpha\!m_k\Big\|^2_{L^2}\lc \|a\|^2_{\ell^2}\sup_k \|m_k\|^2_{L^2_s},
\end{equation*}
by the disjointness of the supports.
\end{proof}

\subsection*{Some reductions}
We remark that it suffices to prove   Theorem~\ref{Tdeltat}
only for very small values of $\delta$, as by straightforward estimation
$$
\int_{\bbR^d}\int_{1}^2|\cS^\delta_tf(x)|^2\frac{dt}{t}\, w(x)\,dx
\lc \delta^{-C} \int_{\bbR^d} |f(x)|^2\,Mw(x) \,dx
$$
for a suitable power $C>0$,  and clearly
$Mw\lc \fM_{\sqrt\delta}w$. In particular we may assume that $\delta$ is small compared to the constant $\eps_2$  in Lemma \ref{scaledbil}.

We may   replace $\cS^\delta_t$  by
\Be\label{Tdeltadef}
T^\delta_t = \psi(D)
\cS^\delta_t
\Ee
where $\psi$ is as in Lemma \ref{scaledbil} (smooth and supported in a ball of radius $\eps_1$).
In view of the  invariance  properties  of the weight operators one can
 use a partition of unity to deduce the weighted inequality
\eqref{unitscale} from the corresponding result for $T^\delta_t$.

We now prepare for an application of Lemma
\ref{scaledbil} and decompose on the frequency side
 the product $F_1\overline{F_2}$ for
suitable $F_i$ (initially  $F_i= T^\delta_t f_i$).
We let $\chi_\circ$ to be a radial $C^\infty_c$ function
$\chi_\circ(\om)=1$ for $|\om|\le 2^{5}$ and so that
 $\supp \chi_\circ$ is contained in $\{\om:|\om|\le 2^5+1\}$, moreover set,
$$\chi_1(\om) =\chi_\circ(\om)-\chi_\circ(2\om).$$
We also let $j_\circ=j_\circ(\delta)$ denote the integer with
$$\sqrt{\delta/\eps_2}\le 2^{-j_\circ}<2\sqrt{\delta/\eps_2}$$
where $\eps_2\in (0,1)$ is as in Lemma \ref{scaledbil}.

Define bilinear forms for pairs of Schwartz functions by
\Be\label{bildef}\begin{aligned} \fB_\circ [F_1,F_2](x)  &=
\frac{1}{(2\pi)^{2d}}\iint \chi_\circ (2^{j_\circ}(\xi-\eta))
\widehat F_1(\xi) \widehat F_2(-\eta)\,e^{\im \inn{\xi-\eta}{x}}
d\xi d\eta,
\\
\cB^j[F_1,F_2](x)  &=
\frac{1}{(2\pi)^{2d}}\iint \chi_1 (2^{j}(\xi-\eta))
\widehat F_1(\xi) \widehat F_2(-\eta)\,e^{\im \inn{\xi-\eta}{x}}
d\xi d\eta.
\end{aligned}\Ee
Then one easily verifies the decomposition
$$
F_1\overline{F_2}= \fB_\circ [F_1,\overline{F_2}] + \sum_{j<j_\circ}
\cB^j[F_1,\overline{F_2}].
$$
Later, in  cases where the supports of the Fourier  transforms of
$F_1$ and $F_2$ are separated we wish to dispense with the frequency
cutoff $\Phi_j(\xi-\eta)$ and replace $\fB_\circ$ or $\cB_j$ by
 a product.
This will be  accomplished by using the identities
\Be\label{averaging}
\begin{aligned}
(2\pi)^{d} \fB_\circ[F_1,F_2](x)&=
2^{-j_0d}  \widehat \chi_\circ (2^{-j_0}\,\cdot\,) *[F_1F_2] (x),
\\
(2\pi)^{d} \cB^j[F_1,F_2](x)&= 2^{-jd}\widehat \chi_1 (2^{-j}\,\cdot\,) *[F_1F_2] (x)
\end{aligned}
\Ee
which follow from
 the Fourier inversion formula (and the assumption that
$\chi_\circ$ is radial).

The desired weighted norm inequality \eqref{unitscale} follows from
the following two propositions, applied for $f_1=f_2=f$.
The proof of the first one is rather straightforward.
\begin{proposition}  \label{Bcirc}
\begin{equation*}
\Big|\int\! \int_{1}^{2}
\fB_\circ [T^\delta_t f_1,\overline{T^{\delta}_t f_2}](x) \frac{dt}{t}\,
w(x)\,dx\Big|
\lc
\delta
\prod_{i=1,2}\Big(\int |f_i(x)|^2 \int \frac{\delta^{d/2}
\fM_{\sqrt\delta}w(x-y)}{(1+\delta^{1/2} |y|)^{d+1}} dy\, dx\Big)^{1/2}.
\end{equation*}
\end{proposition}

\noindent More substantial (and relying on \S\ref{biladjrestr}) is

\begin{proposition} \label{Bj} Let $q\in[1,\frac{d+2}{2})$.
Then, for $j<j_\circ$,
\begin{multline*}
\Big|\int \!\int_{1}^{2}
\cB^j [T^\delta_t f_1,\overline{T^{\delta}_t f_2}](x) \frac{dt}{t}\,w(x)\,dx\Big|
\\ \lc \delta^{2-\frac dq} 2^{-2j (\frac{d}{q}-1)}
\prod_{i=1,2}\Big(\max_{|\nu|\le 8}\int |f_i(x)|^2 H_j*\cW^{j-\nu}_{q,\delta,0}w(x)\,dx\Big)^{1/2},
\end{multline*}
where $H_j(x)=2^{-jd}(1+2^{-j}|x|)^{-d-1}$.
\end{proposition}

We now introduce some basic decompositions.
As in the definition \eqref{Qjth} let $\zeta\in C^\infty_c(\bbR^d)$
be supported in $\{y:|y|\le 1/8\}$, now with the additional assumption that
$\zeta (y)=1$ for $|y|\le 1/9$.
Let $\widetilde \zeta\in C^\infty_c(\bbR^d)$
be supported in $\{y:|y|\le 1/7\}$ so that
$\widetilde \zeta (y)=1$ for $|y|\le 1/8$; hence
$\widetilde \zeta\zeta=\zeta$.
 Let $\varphi$
be a smooth function supported in $[-9/8,9/8]$ and equal to 1 on
$[-7/8,7/8]$ such that $$\sum_{n\in\Z}\varphi(\,\cdot\,-n)=1.$$
We let $\Theta_j$ be a maximal $2^{-j-d}$-separated set of $S^{d-1}$ and
define  for $n\in \bbZ$, $l\in \bbZ$,  operators via the Fourier transform by
\begin{align*}
\cF[Q^{\theta,j}_n  f](\xi) &=  \frac{\zeta(2^j(\tfrac{\xi}{|\xi|}-\theta))}
{\sum_{\theta'\in \Theta_j}\zeta^2(2^j(\tfrac{\xi}{|\xi|}-\theta'))}
\vphi(2^j|\xi|-n) \widehat f(\xi),
\\
\cF[P^{\theta,j}_l  f](\xi) &=  \widetilde \zeta(2^j(\tfrac{\xi}{|\xi|}-\theta))
\vphi(2^{2j}|\xi|-l) \widehat f(\xi);
\end{align*}
moreover, with $Q^{\theta,j}$ as in  \eqref{Qjth} set
$$ \cQ^{\theta,j}= \sum_n Q^{\theta,j}Q^{\theta,j}_n$$
so that
$$
T^\delta_t f \,= \,
\sum_{\theta\in  \Theta_j}T^\delta_t \cQ^{\theta,j}f\,=\,
\sum_{\theta\in \Theta_j}
\sum_{n\in \bbZ} \sum_{l\in \bbZ}
T^\delta_t Q^{\theta,j}
Q^{\theta,j}_n
P^{\theta,j}_lf,
$$
for both cases $j<j_\circ$ and $j=j_\circ$.

Note that the multipliers for $Q^{\theta,j}$ and $\cQ^{\theta,j}$ are
 contained in a sector of width $c2^{-j}$ around $\theta$. The multiplier for
$Q^{\theta,j}_n $ is contained in a $c2^{-j}$ ball centered around $\Xi$ with
$|\Xi|=2^{-j}n +O(2^{-j})$.
The multiplier for
$P^{\theta,j}_l$ is contained in a plate with $(d-1)$ sides
of length $O(2^{-j})$ and a short side of length $O(2^{-2j})$, the long sides
being perpendicular to $\theta$ .
Note that for $2^{-2j}m\in [1,2]$
\Be
\label{tnmloc}
\begin{aligned}
&T^\delta_t Q^{\theta,j}  Q^{\theta,j}_n
P^{\theta,j}_l\neq 0  \text{ for some } t\in [(m-1)2^{-2j}, (m+1)2^{-2j}],\quad
\\ &\qquad \implies
|m-l|\lc 1, \quad
|2^{-2j}l-2^{-j} n|\lc 2^{-j}
\,.
\end{aligned}
\Ee

We also notice that,  from the  localization and separation properties
 of the cutoff functions
$\chi_\circ(2^{j_\circ}(\xi-\eta))$ and  $\chi_1(2^{j}(\xi-\eta))$
in \eqref{bildef},
\Be\label{thetasep}
\begin{aligned}
\fB_\circ(T^\delta_t Q^{\theta,j} g, \overline{T^\delta_t Q^{\theta',j} g})\neq 0 \,
&\implies \dist(\theta,\theta')\le C 2^{-j_\circ}\, ,
\\
\cB_j(T^\delta_t Q^{\theta,j} g, \overline{T^\delta_t Q^{\theta',j} g})\neq 0
&\implies \, 2^{-j+3}\le \dist(\theta,\theta')\le 2^{-j+6}
\, .
\end{aligned}\Ee

\subsection*{Proof of Proposition ~\ref{Bcirc}}
We  first observe by a straightforward integration by parts
that for $t\in [1,2]$ the convolution kernel associated with 
$T^{\delta}_tQ^{\theta,j_{\circ}}$
is dominated by
$$C_d \frac{ \delta^{\frac{d+1}{2}}}{(1+\delta|\inn x\theta|+
\delta^{1/2}|x-\inn{x}{\theta}\theta|)^N} \lc
\cK^{\theta,j_{\circ}}_{2^{-j_{\circ}}}(x) \approx
\cK^{\theta,j_{\circ}}_{2^{j_{\circ}}\delta}(x). $$
On the right hand side
we may also replace $\theta$ by any $\widetilde \theta$ with
$|\theta-\widetilde\theta|\lc 2^{-j_\circ}$.
For more compact notation we write
$$\cK^{\theta,j_\circ}(x):=\cK^{\theta,j_{\circ}}_{2^{-j_{\circ}}}(x). $$
Now for $|\theta_1-\theta_2|\lc 2^{-j_\circ}$ and each $t\in [1,2]$,   we have
\begin{equation}
\label{straightf}
\begin{aligned}
&\Big|\int
\prodstar[T^{\delta}_tQ^{\theta_i,j_{\circ}}g_i(x)]
w(x)\,dx\Big|
\lc
\int \Big[\prod_{i=1,2}(\cK^{\theta_i,j_{\circ}}*  |g_i|^2(x))^{1/2}\Big] |w(x)| dx
\\
&\lc
\prod_{i=1,2}\Big( \int
\cK^{\theta_i,j_{\circ}} * |g_i|^2(x)
 |w(x)| dx\Big)^{1/2}
\lc
\prod_{i=1,2}\Big( \int |g_i(y)|^2
\,\cK^{\theta_i,j_{\circ}}
*  |w|(y)\,dy\Big)^{1/2}\,;
\end{aligned}
\end{equation}
here we used $\|\cK^{\theta,j_{\circ}}\|_1=O(1)$ and applied the Schwarz
inequality twice.

By  Lemma ~\ref{carl}
\Be\label{pjlest}\sum_l \big |P^{\theta,j_\circ}_l f(x)|^2 \lc
\cK^{\theta_i,j_{\circ}} *|f|^2 (x).
\Ee
Now let  $I_m^{j_\circ}=[2^{-2j_\circ}m,
2^{-2j_\circ}(m+1)].$ It will be implicit in all $m$-summations that
$I_m^{j_\circ}\subset[1,2]$.
In what follows $\fA^m_{j_\circ}$ will be an index set consisting of
$(l_1,l_2,n_1,n_2)$ with
$|l_i-m|\lc 1$, $|2^{-2j_{\circ}}l_i-2^{-j_\circ}n_i|\lc 2^{-j_{\circ}}$ for $i=1,2$.
 Then
\begin{align*}
&\Big|
\sum_{\substack{\theta_1, \theta_2\in\Theta_j\\ |\theta_1-\theta_2|\le C2^{-j_\circ}}}
\int \int_1^2\prodstar[T^{\delta}_t \cQ^{\theta_i,j_\circ}f_i(x)]
\frac{dt}{t}w(x)\,dx\Big|
\\
=\ &
\Big|\sum_{\substack{\theta_1, \theta_2\in\Theta_j\\ |\theta_1-\theta_2|\le C2^{-j_\circ}}}
\sum_{m}
\int_{I^j_m}
\sum_{\substack{(l_1,l_2, n_1,n_2)\\ \in \fA^m_{j_\circ}} }
\int \prodstar[
T^{\delta}_t Q^{\theta_i,j_\circ}P^{\theta_i,j_\circ}_{l_i}Q^{\theta_i,j_\circ}_{n_i}
f_i(x)] w(x)\,dx \frac{dt}{t}
\Big|
\\
\lc\ & 2^{-2j_\circ}
\sum_{\substack{\theta_1, \theta_2\in\Theta_{j_\circ}\\ |\theta_1-\theta_2|\le C2^{-j_\circ}}}
\sum_{m}\sum_{\substack{(l_1,l_2, n_1,n_2)\\ \in \fA^m_{j_\circ}} }
\prod_{i=1,2}\Big( \int |
P^{\theta_i,j_\circ}_{l_i}Q^{\theta_i,j_{\circ}}_{n_i}
f_i(x)|^2
\,\cK^{\theta_i,j_{\circ}}
*  |w|(x)\,dy\Big)^{1/2};
\end{align*}
here we have applied \eqref{straightf} and carried out the
$t$ integration.
We now notice that
$|n_1-n_2|\lc 1$, $|l_1-l_2|\lc 1$ for
$(l_1,l_2,n_1,n_2)\in \fA^m_{j_\circ}$
and that for fixed
$(l_1,l_2,n_1,n_2)$ there are only $O(1)$ integers $m$ for which
$(l_1,l_2,n_1,n_2)\in \fA^m_{j_\circ}$.
Hence,  by various applications of the Schwarz inequality  and then
by \eqref{pjlest}
the last displayed quantity  is controlled by
\begin{align}
&2^{-2j_\circ} \prod_{i=1,2}
\Big(\sum_{\theta_i\in\Theta_{j_\circ}}\sum_{l_i}\sum_{n_i}\int
\big|P^{\theta_i,j_\circ}_{l_i}Q^{\theta_i,j_{\circ}}_{n_i}
f_i(x)\big |^2
\,\cK^{\theta_i,j_{\circ}}
*  |w|(x)\,dx\Big)^{1/2}
\notag
\\
\lc\ & 2^{-2j_\circ} \prod_{i=1,2}
\Big(\sum_{\theta_i\in\Theta_{j_\circ}}\sum_{n_i}\int
\big|Q^{\theta_i,j_{\circ}}_{n_i}
f_i(x)|^2
\,\cK^{\theta_i,j_{\circ}} *\cK^{\theta_i,j_{\circ}}
*  |w|(x)\,dy\Big)^{1/2}
\label{withQjthm}
\end{align}
where, by Lemma~\ref{prelptwise}, we may replace
$\cK^{\theta_i,j_{\circ}} *\cK^{\theta_i,j_{\circ}} $
with
$\cK^{\theta_i,j_{\circ}} $.

By Lemma~\ref{carl} we have, with  $H_{j_\circ}(x)=2^{-j_\circ d}
(1+2^{-j_\circ}|x|)^{-d-1}$,
$$
\sum_{\theta\in\Theta_{j_\circ}}\sum_{n}
\big|Q^{\theta,j_{\circ}}_{n} f|^2 \lc H_{j_\circ}*|f|^2 (x)
$$
so that we may estimate
the term \eqref{withQjthm}
by
\Be
2^{-2j_\circ} \prod_{i=1,2}
\Big(\int
|f_i(x)|^2
\,H_{j_\circ} *\sup_{\theta}[
\cK^{\theta,j_{\circ}} *  |w|](x)\,dx\Big)^{1/2}\,.
\Ee

Now by the definition of $\fB_\circ$, \eqref{averaging}
and translation invariance
\begin{align*}\fB_\circ [T^\delta_t f_1,\overline{T^{\delta}_t f_2}](x)
&=\sum_{\substack{\theta_1, \theta_2\in\Theta_{j_\circ}\\ |\theta_1-\theta_2|\le C2^{-j_\circ}}}
\fB_\circ [T^\delta_t \cQ^{\theta_1,j_\circ}f_1,
\overline{T^{\delta}_t\cQ^{\theta_2,j_\circ} f_2}](x)
\\&=\sum_{\substack{\theta_1, \theta_2\in\Theta_{j_\circ}\\ |\theta_1-\theta_2|\le C2^{-j_\circ}}}
\frac{1}{(2\pi)^{d}}\int 2^{-j_\circ d}\widehat \chi_\circ (2^{-j_\circ}h)
\prodstar [T^\delta_t \cQ^{\theta_i,j_\circ}\tau_hf_i] (x)\, dh
\end{align*}
where $\tau_h f(x)=f(x-h)$.
We combine this identity  with the previous estimate
and the obvious inequality   $|2^{-j_\circ d}\widehat \chi_\circ (2^{-j_\circ}\cdot)|\lesssim H_{j_\circ}$
 to obtain
\begin{align*}
&\Big|\int \int_{1}^{2}
\fB_\circ [T^\delta_t f_1,\overline{T^{\delta}_t f_2}](x) \frac{dt}{t}\,
w(x)\,dx\Big|
\\ \lc\ & \int H_{j_\circ}(h)
\Big|\int \int_{1}^{2}\sum_{\substack{\theta_1, \theta_2\in\Theta_{j_\circ}\\ |\theta_1-\theta_2|\le C2^{-j_\circ}}}
\prodstar [T^\delta_t \cQ^{\theta_i,j_\circ}\tau_hf_i](x)
\frac{dt}{t}\,
w(x)\,dx\Big| \, dh
\\ \lc\ & \int H_{j_\circ}(h)
2^{-2j_\circ} \prod_{i=1,2}
\Big(\int
|f_i(x-h)|^2
\,H_{j_\circ} *\sup_{\theta}[\cK^{\theta,j_{\circ}}
*  |w|](x)\,dx\Big)^{1/2}
dh \,.
\end{align*}
 Clearly $\sup_{\theta}[\cK^{\theta,j_{\circ}} *|w|]\lc
\fM_{\sqrt\delta}[w].$
By the  Schwarz inequality and a subsequent
change of variable the last displayed quantity is bounded  by
$$
C\delta \prod_{i=1,2}
\Big(\sum_{\theta_i\in\Theta_{j_\circ}}\int
|f_i(x)|^2
\,H_{j_\circ}*H_{j_\circ} *\fM_{\sqrt\delta}[w](x)\,dx\Big)^{1/2}\,.
$$
Since $H_{j_\circ}*H_{j_\circ}(x)\lc \delta^{d/2}(1+\delta^{1/2}|x|)^{-d-1}$,
by Lemma~\ref{prelptwise},
  this concludes the proof of the proposition.
\qed

\subsection*{Proof of Proposition ~\ref{Bj}}
We fix $j<j_\circ$ so that $2^{2j}\delta < 1/2$.
We define that  $\theta\sim \theta'$ if $\theta\in \Theta_j$,
$\theta'\in \Theta_j$ and $2^{-j+3}\le |\theta-\theta'|\le 2^{-j+6}$; this is the relevant range in \eqref{thetasep}.

Below we shall prove the estimate
\begin{multline} \label{wopsi}
\Big|\int\int_1^2\prodstar [T^\delta_t \cQ^{\theta_i,j} f_i(x)]
 \frac{dt}{t}\,
w(x) \, dx\Big|
\\ \lc
\delta^{2-\frac dq} 2^{-2j (\frac{d}{q}-1)}\prod_{i=1,2} \Big(\sum_n\int
|Q^{\theta_i,j}_nf_i(x)|^2  (\cK^{\theta_i,j}_{2^j\delta}* |w|^q(x))^{1/q} dx\Big)^{1/2}
\end{multline}
for any  $\theta_1, \theta_2$ with $\theta_1\sim\theta_2$.
We first show why \eqref{wopsi}  implies the asserted estimate.

It is  crucial  to observe that for $\theta_1\sim\theta_2$ and
 any $g_1,g_2$
the Fourier transform of the product
$T^\delta_t Q^{\theta_1,j}g_1\, \overline{T^\delta_t Q^{\theta_2,j}g_2}$
 is supported in
$$\{\xi: 2^{-j+2}\le |\xi|\le 2^{-j+6}, \,\,|\inn\xi \vth|\le 2^{-2j+12}\}\,;$$
here $\vth=\frac{\theta_1+\theta_2}{|\theta_1+\theta_2|}$.
Let $\eta_0\in C^\infty_c(\bbR)$ be even and  supported in $[-2^{13},2^{13}]$
so that $\eta_0(s)=1$ for $|s|\le 2^{12}$ and let
$\eta_1\in C^\infty_c(\bbR)$ be supported in $[2, 128]$ so that  $\eta_1(s)=1$ for $s\in [4, 64]$.
Consider the even  Schwartz function $\Psi\equiv \Psi^{j,\theta_1,\theta_2}$ defined by
$$\widehat {\Psi^{j,\theta_1,\theta_2}}(\xi)= \eta_0(2^{2j}\inn {\xi}{\vth})\eta_1(2^j|\xi|).$$
Note that there is a constant $C$ (independent of $j,\theta_1, \theta_2$) so that
$$C^{-1}\Psi^{j,\theta_1,\theta_2} \in \cS^{\theta_i, j}_{j-\nu}$$
for  $\nu=0,\dots, 8$, $i=1,2$.

From these considerations it follows that
$$\int \prodstar[T^\delta_t \cQ^{\theta_i,j}f_i](x)
w(x)\,dx
=\sum_{|\nu|\le 8} \int \prodstar[T^\delta_t \cQ^{\theta_i,j}f_i](x)
\Psi^{j,\theta_1,\theta_2}\!*
P_{\nu-j}w(x) \,dx$$
and therefore, for any  $\theta_1, \theta_2$ with $\theta_1\sim\theta_2$,
 \eqref{wopsi} can be changed to
\Be\label{wopsimod}
\Big|\int\int_1^2\prodstar [T^\delta_t \cQ^{\theta_i,j} ](x)  \frac{dt}{t}\,
w(x)\,dx\Big|
\lc
\prod_{i=1,2} \Big(\sum_n\int |Q^{\theta_i,j}_nf_i(x)|^2
U_{q,\delta}^{\theta_i,j} w(x)
\,dx\Big)^{1/2},
\Ee
where
$$U_{q,\delta}^{\theta,j  } w(x)=
\delta^{2-\frac{d}{q}}2^{-2j(\frac dq-1)}
\sum_{|\nu|\le 8}\big(\cK^{\theta,j}_{2^j\delta}*
\sup_{\Psi\in \cS^{\theta,j}_{\nu-j}}|\Psi*P_{\nu-j}w|^q(x)\big)^{1/q}\,.$$

By the definition of $\cB_j$ we have
\begin{align*}
\cB^j [T^\delta_t f_1,\overline{T^{\delta}_t f_2}](x)
&=\sum_{\theta_1\sim\theta_2}
\cB^j [T^\delta_t \cQ^{\theta_1,j}f_1,\overline{T^{\delta}_t\cQ^{\theta_2,j} f_2}](x)
\\
&=\sum_{\theta_1\sim\theta_2}(2\pi)^{-d} \int_h 2^{-jd}\widehat \chi_1(2^{-j}h)
\prodstar [T^\delta_t \cQ^{\theta_i,j}\tau_hf_i](x) dh\,.
\end{align*}
Therefore \eqref{wopsimod} yields
\begin{align*}
&\Big|\int\int_1^2 \cB^j [T^\delta_t f_1,\overline{T^{\delta}_t f_2}](x)
 \frac{dt}{t}\,w(x) \,dx\Big|
\\
\lc\ &
\int H_j(h) \prod_{i=1,2} \Big(\sum_{\theta_i\in \Theta_j}
\sum_n\int |Q^{\theta_i,j}_nf_i(x-h)|^2
U_{q,\delta}^{\theta_i,j} w(x)
\,dx\Big)^{1/2}
\\
\lc\ &
\prod_{i=1,2} \Big(\sum_{\theta_i\in \Theta_j}\sum_n\int |Q^{\theta_i,j}_nf_i(x)|^2
H_j*[\sup_{\theta}  U_{q,\delta}^{\theta,j} w](x)
\,dx\Big)^{1/2},
\end{align*}
by the Schwarz inequality.
Now, for $|\nu|\le 8$,  observe $\cK^{\theta, j}_{2^j\delta}\approx \cK^{\theta, j-\nu}_{2^{j-\nu}\delta}$
and therefore
$$\sup_{\theta}  U_{q,\delta}^{\theta,j} w(x)
\lc \delta^{2-\frac{d}{q}}2^{-2j(\frac dq-1)}
\cW^{j-\nu}_{q,\delta,0}w(x).$$ Moreover, by Lemma~\ref{carl},
$$\sum_{\theta\in \Theta_j}\sum_n|Q^{\theta,j}_nf(x)|^2  \lc H_j*|f|^2(x)\,.$$
Hence, using the Schwarz inequality again, we get
\begin{multline*}
\Big|\int\int_1^2 \cB^j [T^\delta_t f_1,\overline{T^{\delta}_t f_2}](x)
 \frac{dt}{t}\,w(x)\, dx\Big|
\\
\lc \delta^{2-\frac{d}{q}}2^{-2j(\frac dq-1)}
\max_{|\nu|\le 8}\prod_{i=1,2} \Big(\int|f_i(x)|^2 H_j*H_j*
\cW^{j-\nu}_{q,\delta,0}w(x)\, dx\Big)^{1/2}
\end{multline*}
and since $H_j*H_j\lc H_j$ we obtain the asserted estimate.

\subsubsection*{Proof of \eqref{wopsi}}
We argue as in the proof of Proposition~\ref{Bcirc} and rely on
Lemma~\ref{scaledbil}.
We let $I_m^j=[2^{-2j}m,2^{-2j}(m+1)]$ if this interval is a subset of $[1,2]$ (otherwise $I^j_m=\emptyset$).
Define the index sets
$\fA^m_j$ as in the proof of Proposition~\ref{Bcirc} (with $j$ instead of $j_\circ$).

Then the right hand side of \eqref{wopsi} is equal to
\begin{align*}
\Big|\sum_m \sum_{(l_1,l_2,n_1,n_2)\in \fA_j^m}
\int\int_{I^j_m}\prodstar [T^\delta_t Q^{\theta_i,j} Q^{\theta_i,j}_{n_i} P^{\theta_i,j}_{l_i} f_i(x)]  \frac{dt}{t}\,w(x) \, dx\Big|
\end{align*}
 and, by Lemma~\ref{scaledbil}, this is  estimated by
\begin{align*}&\sum_m \sum_{\substack {|l_1-m|\lc 1\\|l_2-m|\lc 1}} \sum_{|n_1-n_2|\lc 1}\delta^{2-\frac{d}{q}}2^{-2j(\frac dq-1)}
\prod_{i=1,2}\Big(
\int\big| P^{\theta_i,j}_{l_i} Q^{\theta_i,j}_{n_i} f_i(x)\big|^2
\big(\cK^{\theta_i, j}_{2^j\delta} *|w|^q(x)\big)^{1/q}  \, dx\Big)^{1/2}
\\
&\,\lc\, 
\delta^{2-\frac{d}{q}}2^{-2j(\frac dq-1)}\prod_{i=1,2}\Big(\sum_{n_i}
\sum_{l_i}\int
\big| P^{\theta_i,j}_{l_i} Q^{\theta_i,j}_{n_i} f_i(x)\big|^2
\big(\cK^{\theta_i,j}_{2^j\delta} *|w|^q(x)\big)^{1/q}  \, dx\Big)^{1/2}
\\
&\,\lc\, 
\delta^{2-\frac{d}{q}}2^{-2j(\frac dq-1)}\prod_{i=1,2}\Big(\sum_{n_i}
\int\big| Q^{\theta_i,j}_n f_i(x)\big|^2
\cK^{\theta_i,j}_{2^{-j}}*\big(\cK^{\theta_i,j}_{2^j\delta} *|w|^q(x)
\big)^{1/q}
\, dx\Big)^{1/2}\,;
\end{align*}
here we have used the  Schwarz inequality  and the bound
$$\sum_l| P^{\theta,j}_l g(x)|^2 \lc  \cK^{\theta, j}_{2^{-j}}* |g|^2(x) $$
which is a consequence of Lemma~\ref{carl}.
We also have
$$\cK^{\theta, j}_{2^{-j}}*
\big(\cK^{\theta, j}_{2^j\delta} *|w|^q(x)\big)^{1/q}
\lc \big(\cK^{\theta, j}_{2^{-j}}*
\cK^{\theta, j}_{2^j\delta} *|w|^q(x)\big)^{1/q}
\lc \big(\cK^{\theta, j}_{2^j\delta} *|w|^q(x)\big)^{1/q} ;
$$
where the first estimate  follows from H\"older's inequality and the
second from \eqref{cKconv} and the assumption  $2^j\delta<2^{-j}$.
This completes the proof of \eqref{wopsi} and thus the proposition is established.
\qed


\section{$L^p(L^2)$ estimates for solutions of Schr\"odinger and wave equations}
\label{wschrsect}
\begin{proposition}\label{wavefrloc}
Let $d\ge 2$ and  $p\in(\frac{2(d+2)}{d},\infty]$, or $d=1$ and $p=\infty$, and let $a\in(0,\infty)$.
Let $I$ denote a compact interval of time. Then for $k\ge 1$,
\begin{equation}\label{psaloc}
\Big\|\Big(\int_I|U^a_t P_kf|^2 \,dt\Big)^{1/2}\Big\|_{p}
\lc 2^{k a \la(p)}\, \|f\|_{L^p},\quad \la(p)=d\Big(\frac 12-\frac 1p\Big)-\frac 12.
\end{equation}
\end{proposition}

\begin{proof} [Proof of Theorem~\ref{wavey}]
The result is an immediate consequence of
Proposition~\ref{wavefrloc} and the case  $q=2$, $r=1$, $\wp=p$
of
Proposition ~\ref{combin}.
\end{proof}

\begin{proof} [Proof of Proposition~\ref{wavefrloc}]
We may assume that  $2^k$ is large.
Let $\phi$ be an even real-valued  function
so that
$\widehat \phi(t)>c>0$ for $t\in I$ and $\supp \phi\subset [-1/4,1/4]$.
It suffices to estimate the $L^p(\bbR^d,L^2(\bbR))$ norm of
$\widehat\phi(t)U^a_t P_kf(x)$.
For fixed $x$ the $L^2(\bbR)$ norm of this expression is equal to
\begin{align*}
&\frac{1}{(2\pi)^{1/2}}
\Big(\int_{\bbR}\Big| \int \widehat \phi(t) \exp(-\im t\tau)
U^a_t P_kf(x) \,dt\Big|^2 d\tau\Big)^{1/2}
\\
=\ &\frac{1}{(2\pi)^{1/2}} \Big(\int_{2^{(k-3)a}}^{2^{(k+3)a}}\Big| \cF^{-1} \big[
\phi(|\,\cdot\,|^a-\tau) \chi(2^{-k}|\,\cdot\,|)\widehat f \,\big](x)\Big|^2 d\tau \Big)^{1/2}.
\end{align*}
By a finite splitting
we may replace
the integral over $[2^{(k-3)a}, 2^{(k+3) a}]$ by an integral over $((2^kT)^a,
(2^{k+1}T)^a)$ with $T\approx 1$. After changing variables $\tau=(2^kT r)^a$ it suffices to
show that
$$
\Big\| \Big(\int_{1}^2\Big| \cF^{-1} \big[
\phi(|\,\cdot\,|^a-(2^kTr)^a) \chi(2^{-k}|\,\cdot\,|)\widehat f \,\big]\Big|^2 dr\Big)^{1/2}\Big\|_p\lc 2^{a(\la(p)-\frac 12) k} \|f\|_{p},
$$
or, after scaling and setting $\delta= (2^{k}T)^{-a}$,
\Be\label{BRsqfct}
\Big\| \Big(\int_{1}^2\Big| \cF^{-1} \big[
\phi( \delta^{-1} (|\,\cdot\,|^a-r^a))
\chi(T|\,\cdot\,|)\widehat f\,\big]\Big|^2 dr\Big)^{1/2}\Big\|_p\lc \delta^{\frac 12-\la(p)} \|f\|_{p}\,.
\Ee
But as $\phi$ is even,
$\phi( \delta^{-1} (|\xi|^a-r^a))  =\phi\big(\delta^{-1} \beta(\xi,r)
(1-|\xi|^2/r^2)\big)
$ where $\beta(\xi,r)= r^2\frac{r^a-|\xi|^a}{r^2-|\xi|^2}$ is smooth for $\xi$ away from the origin, and nonvanishing.
Thus Corollary~\ref{Tdeltatcor}
may be applied and  we get the $L^p$ inequality
\eqref{BRsqfct} for $d\ge 2$ and $p>2 +4/d$.

The case $d=1$, $p=\infty$ is more straightforward; the estimate
\Be\label{BRsqfctinfty}
 \Big(\int_{1}^2\big| \cF^{-1} \big[
\phi( \delta^{-1} (|\,\cdot\,|^a-r^a))
\chi_1(T|\,\cdot\,|)\widehat f\,\big](x)\big|^2 dr\Big)^{1/2}\lc \delta^{1/2}
 \|f\|_{\infty}
\Ee
for $T\approx 1$ follows from
$$
\sum_{0\le n< \delta^{-1}}
\big| \cF^{-1} \big[
\phi( \delta^{-1} (|\,\cdot\,|^a-(1+n\delta+\sigma)^a)
\chi_1(T|\,\cdot\,|)\widehat f\,\big](x)\big|^2 \lc
 \|f\|_{\infty}^2, \qquad 0<\sigma\le \delta,
$$
and integration in $\sigma$. The last displayed inequality however
is a consequence of   Lemma~\ref{carl}.
\end{proof}

We finish by stating a global variant of the one-dimensional square function estimate which does not use Sobolev spaces and which we will not use elsewhere in the paper.

\begin{proposition}\label{combin2} Let $d=1$, $p\in[2,\infty),$ $a\in(0,\infty)$,  and let $I$ be a compact interval.
Then
$$
\Big \|\Big(\int_I|U^a_tf|^2\,dt\Big)^{1/2}\Big\|_{L^p(\bbR)}
\lc\, \|f\|_{L^p(\bbR)}.
$$
Moreover
$$\Big \|\Big(\int_I|U^a_t f|^2\, dt\Big)^{1/2}
\Big\|_{BMO(\bbR)} \lc\|f\|_{L^\infty(\bbR)}.
$$
\end{proposition}
Given the reduction in the localized case, in the proof of Proposition
\ref{wavefrloc}, the $L^p(L^2)$ estimates can be deduced
from a regularized version of
Rubio de Francia's square function estimate~\cite{ru1}
associated with
arbitrary disjoint  collection of intervals.
The $L^\infty$-$BMO$ estimate can be obtained from Sj\"olin's proof
\cite{sj2}  of that estimate. We omit the details.

\section{An $L^p(L^q)$ estimate}\label{LpLqsect}

We state the
$L^p(L^q)$ estimates alluded to in the introduction.
We work with the norm
$$\|u\|_{L^p(L^q(I))}
= \Big\|\Big(\int_I|u(\,\cdot\,,t)|^q dt\Big)^{1/q}\Big\|_{L^p(\bbR^d)}$$
in $L^p(\bbR^d;L^q(I)),$ with  the usual  modification
$\|u\|_{L^p(L^\infty(I))}
= \|\sup_{t\in I}|u(\,\cdot\,,t)|\|_p$ if $q=\infty$.

\begin{theorem} \label{Schrthm}
Let  $a\in(1,\infty)$
and let $I$ be a compact interval of time. Then
\begin{equation}\label{LpLqschr}
\big\|U^a f\big\|_{L^p(L^q(I))} \lc \|f\|_{L^p_s}, \quad \frac {s}{a}=
d\Big( \frac 12-\frac 1p\Big)- \frac 1q
\end{equation}
holds true in each of the following three  cases:

(i) $d=1$, $4<p<\infty$, $\frac{2p}{p-2}<q\le \infty$.

(ii) $d\ge 2$, $\frac{2(d+3)}{d+1}<p\le \frac{2(d+2)}d$, $\frac{2p}{(d+1)p-2(d+2)}<q\le\infty$.

(iii)   $d\ge 2$, $\frac{2(d+2)}{d}<p<\infty$, $2\le q\le \infty$.


\end{theorem}

\noi{\it Remark.}  The statements (i) and (iii) also hold for $0<a<1$.

\begin{proof}
The stated  results for $p\le q\le \infty$  are in~\cite{rose}.
Consider the inequality
\begin{equation}\label{LpLqschrdyad}
\sup_{k\in \bbN}2^{-k a( \frac d2-\frac dp- \frac 1q)}
\big\|U^a [P_k f]\big\|_{L^p(L^q(I))} \lc
\|f\|_{L^p}
\end{equation}
which holds for $2(d+3)/(d+1)<p\le q\le \infty$, $d\ge 2$  and
$4<p\le q\le \infty$, $d=1$, by~\cite{rose}. It holds for $q=2$ if $p=\infty$, $d=1$ and $2+4/d<p\le \infty$ if $d\ge 2$, by Theorem~\ref{wavey}.
By complex interpolation
\eqref {LpLqschrdyad} also  holds for
$d\ge 2$,  $\frac{2(d+3)}{d+1}<p\le \frac{2(d+2)}d$ and
$\frac 1q<\frac{d+1}{2}-\frac{d+2}p$ which is equivalent with
$\frac{2p}{(d+1)p-2(d+2)}<q\le\infty$. Moreover for $d=1$ complex interpolation shows that \eqref {LpLqschrdyad}   holds for $\frac 2q<1-\frac 2p$ ({\it i.e.}
$\frac{2p}{p-2}<q\le \infty$). Finally we may combine the dyadic pieces by
using Proposition~\ref{combin} in the appendix.
\end{proof}

\medskip

\noi {\it Remark.} For $a=2$, we obtain further
improvements in~\cite{lerose2}, in particular in two dimensions
 an $L^p(L^4)$ bound for $p>16/5$.

\appendix

\section{Combining  frequency localized  pieces} \label{localtoglobal}
We state a  variant of results by Fefferman and Stein~\cite{fest} and
Miyachi~\cite{mi}  which is motivated by its application to prove Theorems
\ref{wavey} and~\ref{Schrthm}. The approach extends and somewhat simplifies the one in~\cite{rose}  (see also~\cite{se},~\cite{prrose}
for related results).
For later applications in \cite{lerose2} we formulate the results in
slightly more generality than needed in this paper
(in particular here we just need the case $p=\wp$ in
Theorem \ref{dyadicpieces}.

Let $\sB$ be a Banach space with norm $|\,\cdot\,|_\sB$; in our application
$\sB=L^q(I)$ for a compact interval $I$.
 We consider convolution operators $T_k$, with $k\in\N$,  mapping $L^1(\bbR^d)$ into $L^1(\bbR^d,\sB)$, a space of $\sB$-valued functions. We define $T_k$ by
$$T_k f(x)=h_k*f(x)=\int h_k(x-y) f(y)\,dy,$$
 where
for each $k$
we make, for simplicity,   the {\it a-priori} assumption that
$h_k\in L^1(\R^d,\sB)$  but we do not assume a bound
on these $L^1$
norms.  We shall be interested in situations where, for some $a>0$, the
 part of the kernel $h_k$ supported in $|x|\ge C_1 2^{k(a-1)}$ can be  neglected.
In particular, this is true of $U_t^a$ as defined in \eqref{Ua}.

In what follows  let  $\rho_k\in C^1(\bbR^d)$ be such that
\Be \label{rhok}
\begin{gathered}  |\rho_k(x)|+2^{-k}|\nabla \rho_k(x)|\le 2^{kd},
\\
\supp \rho_k \,\subset \, \{x:|x|\le 2^{-k}\}.
\end{gathered}
\Ee
We define $R_k$ on $\sB$-valued functions $g$ by
\Be \label{Rkdef}
 R_k g(x)= \rho_k* g(x)=\int\rho_k(y) g(x-y)\,dy. \Ee
In applications the   operators $R_k$ often arise from
  dyadic frequency decompositions, however no cancellation
condition on $\rho_k$ is needed in the following result.

\begin{theorem}  \label{dyadicpieces}
Let $\wp_0\in(1,\infty)$, $\wp_0\le p_0$, $1/\wp_0-1/p_0=1/\wp_1$
 and $a\in(0,\infty)$. With $T_k$ and $R_k$  defined as above, let
\Be\label{Ap0} A:=\sup_{k>0} 2^{kad/p_0}\|T_k\|_{L^{\wp_0}\to L^{p_0}(\sB)}
\Ee
and, for $1/\wp_1+1/\wp_1'=1$,
\Be \label{smallness}
B:=
\sup_{k>0} 2^{kad/p_0}    \Big(\int_{|x|\ge C_1 2^{k(a-1) }}|h_k(x)|_{\sB}^{\wp_1} \,dx\Big)^{1/\wp_1'}
\Ee
for some fixed constant $C_1\ge 1$.
Then for all $p\in (p_0,\infty)$ and  $r>0$,   there exists $C=C(p_0,p,r,C_0,C_1,d)$
 so that
\begin{equation}\label{trlizest}
\Big\|\,\Big( \sum_{k>0} 2^{kad r/p}
|R_k{T}_{k} f_{k}|_{\sB}^r\Big)^{1/r}
 \,\Big\|_{p} \le C
A \Big(1+\frac{B}{A}\Big)^{1-p_0/p}
\Big(\sum_{k>0} \|
f_{k}\|_{\wp}^p\Big)^{1/p}\,, \quad
\frac 1{\wp}-\frac 1{p} =\frac 1{\wp_0}-\frac 1{p_0}.
\end{equation}
Moreover,
\begin{equation}\label{bmoest}
\Big\|\,\Big( \sum_{k>0}
|R_k{T}_{k} f_{k}|_{\sB}^r\Big)^{1/r}
 \,\Big\|_{BMO} \le C
\big(A+B\big)
\sup_{k>0} \|
f_{k}\|_{\wp_1}\, .
\end{equation}
\end{theorem}

Before we begin with the proof we state a preliminary lemma.

\begin{lemma}\label{pversp0}
Define
$$\cT_k f(x)= 2^{kd}\int_{|x-y|\le 2^{-k}}|T_k f(y)|_\sB dy$$
and,  with the notation as in \eqref{Ap0} and \eqref{smallness}, let
 \Be \label{asp}
\cA(p)=
C_1^{d(1/p_0-1/p)} A + A^{{p_0}/{p}}
B^{1-{p_0}/{p}} \,.
\Ee
Then, for $p_0\le p\le \infty$,
$$
\|\cT_k f\|_{L^p(\sB)} \lc 2^{-kad/p}
\cA(p)
\|f\|_\wp\, \quad
\frac 1{\wp}-\frac 1{p} =\frac 1{\wp_0}-\frac 1{p_0}\,.
$$
\end{lemma}

\begin{proof}
We interpolate between $p=p_0$ and $p=\infty$.
Since $\cA(p_0)=A$ the inequality is immediate for $p=p_0$
from assumption \eqref{Ap0}.

To prove the inequality for $p=\infty$ we choose a grid $\fQ^k_a$ of
cubes $\mathcal{Q}$  of
sidelength $2^{k(a-1)}$,
so that the cubes in $\fQ^k_a$ have  disjoint interior
and  $\sum_{\mathcal{Q}\in \fQ^k_a}\chi_\mathcal{Q}=1$ almost everywhere. For  each $\cQ\in \fQ^k_a$  let
$\mathcal{Q}^*$ be the cube with
same center as $\mathcal{Q}$ and sidelength $10 d  C_1 2^{k(a-1)}$.
It then  suffices to show that for each
cube $\cQ\in \fQ^k_a$
\Be \label{Linftybd}
\chi_\cQ(x)| \cT_k f(x)|_\sB \lc \big(C_1^{d/p_0} A+B\big) \, \|f\|_\infty
\Ee
for every $x$.
Given $\mathcal{Q}$ we split $f= b_\cQ+g_\cQ$ where $b_\cQ=f\chi_{\cQ^*}$ and $g_\cQ= f\chi_{\bbR^d\setminus \cQ^*}$.
For $b_\cQ$ we apply H\"older's inequality and use assumption \eqref{Ap0}, so that
\begin{align*}
|\cT_k b_\cQ(x)|_\sB &\lc 2^{kd/p_0} \big\||T_k b_\cQ|_\sB\big\|_{p_0} \lc
2^{kd/p_0} A 2^{-kad/p_0}
\|b_\cQ\|_{\wp_0}
\\&\lc  A  2^{kd/p_0} 2^{-kad/p_0} |\cQ^*|^{1/\wp_0-1/\wp_1} \|f\|_{\wp_1}
\,\lc\,C(d)  C_1^{d/\wp_0-d/\wp_1}A\,\|f\|_{\wp_1}
\end{align*}
since $|\cQ^*| \approx C_1^d 2^{k(a-1)d}$.

For $g_\cQ$ we note that when $x\in \cQ$, $w\notin \cQ^*$ and $|x-z|\le 2^{-k}$, then $|z-w|\ge C_12^{k(a-1)}$, so we can  use \eqref{smallness} to estimate
\begin{align*}
|\cT_k g_\cQ(x)|_\sB &\le 2^{kd}\int_{|x-z|\le 2^{-k}}\,
\int_{|z-w|\ge C_12^{k(a-1)}} |h_k(z-w)|_{\sB}|f(w)|\, dw \, dz
\\&\lc B\,\|f\|_{\wp_1}
\end{align*}
Combining the two estimates, we  get \eqref{Linftybd}.
\end{proof}

\begin{proof}[Proof of Theorem~\ref{dyadicpieces}]
We may assume $r\le 1$ and that the summation in  $k$ is extended over a finite set. We proceed as in~\cite{prrose} and,
by the Fefferman-Stein theorem~\cite{fest}  on the $\#$-maximal operator and the inequality $|\,|u|_{\sB}^r-|v|_{\cB}^r|\le |u-v|_{\cB}^r$
we get
\begin{align*}
&\Big\|\Big(\sum_k |2^{kad/p} R_kT_kf_k|^r_{\sB} \Big)^{1/r}\Big\|_{p}
\\&\lc\Big\|\sup_{Q:x\in Q} \sum_k 2^{kad r/p}
\intslash_Q \intslash_Q
| R_kT_kf_k(y)-R_kT_kf_k(z)|_\sB^r \, dz \,dy
\Big\|_{L^{p/r}(dx)}^{1/r}\,.
\end{align*}
Let $x\mapsto Q(x)$ depend measurably on $x$, so that for each $x$
the cube $Q(x)$ is centered at $x$ and has sidelength   in
$[2^{L(x)}, 2^{L(x)+1})$.
It suffices to estimate the $L^p$ norm of
$$\Big( \sum_k 2^{kad r/p}
\intslash_{Q(x)} \intslash_{Q(x)}
| R_kT_kf_k(y)-R_kT_kf_k(z)|_\sB^r \, dz \,dy\Big)^{1/r}\,.$$
We let  $F=\{f_k\}_{k>0}$
and estimate the displayed expression by
$\sum_{i=1}^3\fS_i F(x)$
 where
\begin{align*}
\fS_1F(x)&=\Big(\sum_{k+L(x)\le 0} 2^{kad r/p}
\intslash_{Q(x)} \intslash_{Q(x)}
| R_kT_kf_k(y)-R_kT_kf_k(z)|_\sB^r \, dz \,dy
\Big)^{1/r}\,,
\\
\fS_2F(x)&=\Big(\sum_{\substack{ k+L(x)> 0\\ k(a-1)\le L(x)}}
 2^{kad r/p}
\intslash_{Q(x)}
| R_kT_kf_k(y)|_\sB^r \,dy
\Big)^{1/r}\,,
\\
\fS_3F(x)&=\Big(\sum_{\substack{ k+L(x)> 0\\ k(a-1)> L(x)}}
 2^{kad r/p}
\intslash_{Q(x)}
| R_kT_kf_k(y)|_\sB^r \,dy
\Big)^{1/r}\,.
\end{align*}
Set $\|F\|_{\ell^p(L^\wp)}=
(\sum_k\|f_k\|_\wp^p)^{1/p}$  for $p<\infty$
and
$\|F\|_{\ell^\infty(L^{\wp_1})}= \sup_k \|f_k\|_{\wp_1}$.
For $p>p_0$  we will
bound  the $L^p$ norms of
$\fS_i F$ by $C\cA(p)\|F\|_{\ell^p(L^\wp)}$.
In the proofs we shall use the notation
$$\om_k(x)= 2^{kd} \chi_{\{|x|\le \sqrt d 2^{-k+3}\}}(x).$$


\subsubsection*{{\it $L^p$ bound for $\fS_1(F)$}}
Using the estimate \eqref{rhok} for   $\nabla \rho_k$ we see that
for $y,z\in Q(x)$
$$| R_kT_kf_k(y)-R_kT_kf_k(z)|_\sB \, \lc 2^{L(x)+k}
\int\om_k(x-u)
|T_kf_k(u)|_\sB du\,.
$$
Using the embedding $\ell^{p/r}\hookrightarrow \ell^\infty$ we estimate
\begin{align*}
\fS_1F(x)^r&\lc \sup_L \Big[ \sum_{k+L\le 0} 2^{k+L}
\int\om_k(x-u)  |2^{kad/p} T_kf_k(u)|_\sB du
\Big]^r
\\
&\lc \Big(\sum_L \Big[ \sum_{0<k\le -L} 2^{k+L}
\int\om_k(x-u) |2^{kad/p} T_kf_k(u)|_\sB du
\Big]^p\Big)^{r/p}
\end{align*} and
therefore, with the change of summation variable $n=-L-k$ and Minkowski's inequality
\begin{align*}
&\|\fS_1F\|_p= \big\|[\fS_1 F ]^r\big\|_{p/r}^{1/r}
\\&\lc
\sum_{n>0} 2^{-n}
\Big(\sum_{L> -n}
\Big\| {2^{-(L+n)d}}\int_{|y|\le C 2^{L+n}}
2^{-(L+n)ad/p} |T_{-L-n} f_{-L-n}(\,\cdot\,-y)|_\sB \, dy\Big\|_p^p \Big)^{1/p}.
\end{align*}
By Lemma
\ref{pversp0}
this can be estimated  by
$$\cA(p) \sum_{n>0} 2^{-n} \Big(\sum_{L>-n}\|f_{-L-n}\|_\wp^p\Big)^{1/p} \lc \cA(p)\|F\|_{\ell^p(L^\wp)}.
$$

\subsubsection*{{\it $L^p$ bound for $\fS_2(F)$}}
The $L^p$ bound for $\fS_2F$ follows by interpolating the inequalities
\Be \label{psest}
\begin{aligned}
\|\fS_2F\|_{p_0}&\lc A\|F\|_{\ell^{p_0}(L^{\wp_0})}\,,
\\
\|\fS_2F\|_{\infty}&\lc (C_1^{d/p_0}A+B)\|F\|_{\ell^{\infty}(L^{\wp_1})}\,.
\end{aligned}
\Ee
For the
 $L^{p_0}$ bound, we sum a geometric series, using $p>p_0$, to estimate
$$\Big(\intslash_{Q(x)} \sum_{\substack{k+L(x)> 0\\ k(a-1)\le L(x)}}
2^{kad r/p}|R_kT_k f_k(y)|_\sB^r dy\Big)^{1/r}
\lc \big(M \big[\sup_{k>0} 2^{kad r/p_0}|R_kT_k f_k|_\sB^r\big](x)\big)^{1/r}$$
and by the $L^{p_0/r}$ boundedness of the Hardy--Littlewood operator we get
\begin{align*}\|\fS_2 F\|_{p_0} &\lc \big\|\sup_{k>0}
2^{kad r/p_0}|R_kT_k f_k|_\sB^r\big\|_{p_0/r}^{1/r}=
\big\|\sup_{k>0}
2^{kad /p_0}|R_kT_k f_k|_\sB\big\|_{p_0}
\\&\lc
\Big(\sum_{k>0} \big\|2^{kad /p_0}|R_kT_k f_k|_\sB\big\|_{p_0}^{p_0}\Big)^{1/p_0}
\lc
\Big(\sum_{k>0} \big\|2^{kad /p_0}|T_k f_k|_\sB\big\|_{p_0}^{p_0}\Big)^{1/p_0}
\\&\lc A
\|F\|_{\ell^{p_0}(L^{\wp_0})}\,,
\end{align*}

For the $L^\infty$ bound,   we fix $x$, $Q=Q(x)$, $L=L(x)$ and
let $y_Q=y_{Q(x)}$ be the center of $Q$.
By H\"older's inequality and \eqref{rhok},
$$
\fS_2 F(x)\lc \Big(
\sum_{\substack{k+L>0\\k(a-1)\le L}}
2^{kad r/p} \Big[\intslash_{Q}\int \om_k(y-z)
\big|T_kf_k](z)\big|_{\sB} \, dz\, dy \Big]^r \Big)^{1/r}.
$$
Let $Q^*$ be the  $C_1 2^{10}d$ dilate of $Q$ with respect to $y_Q$.
We may estimate the last displayed expression by
$\Enear+\Efar$ where
\begin{align*}
\Enear &=\Big(
\sum_{\substack{k+L>0\\k(a-1)\le L}}
2^{kad r/p} \Big[\intslash_{Q}\int\om_k(y-z)
\big|T_k[f_k\chi_{Q^{*}}](z)\big|_{\sB} \,dz\,dy \, \Big]^r \Big)^{1/r}\, ,
\\
\Efar&
=\Big(\sum_{\substack{k+L>0\\k(a-1)\le L}}
2^{kad r/p} \Big[\intslash_{Q}\int\om_k(y-z)
\big|T_k[f_k\chi_{\bbR^d\setminus Q^{*}}](z)\big|_{\sB} \, dz\, dy
\, \Big]^r \Big)^{1/r}\, ,
\end{align*}
and it suffices to check that
\begin{align}
\label{insideQ}
\Enear \, &\lc \,   C_1^{d/p_0}
 A
\|F\|_{\ell^\infty(L^{\wp_1})}\, ,
\\
\label{outsideQ}
\Efar \, &\lc \,
B \|F\|_{\ell^\infty(L^{\wp_1})}\, .
\end{align}


To prove \eqref{insideQ}
 we apply H\"older's inequality, use $p_0<p$ and assumption \eqref{Ap0}:
\begin{align*}
\Enear &\lc  \Big(\sum_k 2^{kad r/p}   \Big( \frac{C_1^d}{|Q^*|}\int
\big|T_{k} [f_{k}\chi_{Q^*}](z)\big|_\sB^{p_0}
dz\Big)^{r/p_0} \Big)^{1/r}
\\
&\lc C_1^{d/p_0} \sup_k 2^{kad/p_0} |Q^*|^{-1/p_0}
\big\|{T}_{k} [f_{k}\chi_{Q^*}]\big\|_{L^{p_0}(\sB)}
\\
&\lc AC_1^{d/p_0}
\sup_k |Q^*|^{-1/p_0}
\big\|f_{k}\chi_{Q^*}\big\|_{\wp_0}
\lc AC_1^{d/p_0}
\sup_k |Q^*|^{-1/p_0} |Q^*|^{1/\wp_0-1/\wp_1}
\big\|f_{k}\chi_{Q^*}\big\|_{\wp_1}
\\&
\lc AC_1^{d/p_0} \|F\|_{\ell^\infty(L^{\wp_1})}\,.
\end{align*}

To prove   \eqref{outsideQ} we use  assumption \eqref{smallness}.
Note that if $y\in Q$, $|y-z|\le \sqrt{d}2^{-k+3}$
(and $k(a-1)\le L$),
 $w\in  \bbR^d\setminus Q^*$, then $|z-w|\ge  C_1 2^{k(a -1)}$.
Thus
\begin{align*}&\intslash_{Q} \int \om_k(y-z)
\big|T_k[f_k\chi_{\bbR^d\setminus Q^{*}}](z)\big|_{\sB} \, dz\, dy
\\&\lc
\intslash_{Q}\int \om_k(y-z)
\int_{|z-w|\ge C_1 2^{k (a-1)}} |h_k(z-w)|_\sB |f_k(w)| dw
\, dz\, dy
\\
&\lc 2^{-kad /p_0} B\|F\|_{\ell^\infty(L^{\wp_1})}
\end{align*}
and
\eqref{outsideQ} follows.

\subsubsection*{{\it $L^p$ bound for $\fS_3(F)$}}

Let $B_L$ be the ball of radius $10 d 2^L$  centered at the origin.
We may estimate
\begin{align*}
\big\| \fS_3(F)\big\|_p&\lc
\Big\|\sup_L  \frac{\chi_{B_L}}{|B_L|} *\Big[ \sum_{\substack{k+L>0\\k(a-1)>L}}
2^{kad r/p} |R_k{T}_{k} f_{k}|_{\sB}^r\Big]\Big\|_{p/r}^{1/r}
\\
&\lc \Big(\sum_{n>0} \Big\|\sup_{L<(1-a^{-1})n}
\frac{\chi_{B_L}}{|B_L|} *\big[
2^{(n-L)ad r/p} |R_{n-L}{T}_{n-L} f_{n-L}|_{\sB}^r\big]\Big\|_{p/r}\Big)^{1/r}
\end{align*}
by  Minkowski's inequality.
By H\"older's inequality on each ball $B_L$ we see that
 the last expression is dominated by
$$
\Big(\sum_{n>0}\Big\|
\sup_{L<(1-a^{-1})n}
\frac{\chi_{B_L}}{|B_L|} *
2^{(n-L)ad /p} |R_{n-L}{T}_{n-L} f_{n-L}|_{\sB}
\Big\|_p^r\Big)^{1/r}.
$$
Now
for $n>0$ we have
$\chi_{B_L}*\om_{n-L}(x)\lc \chi_{B_{L+1}}(x)$. Thus   we get
$$\big\| \fS_3(F)\big\|_p\lc
\Big(\sum_{n>0} \|\fS_{3,n} F\|_p^r\Big)^{1/r}$$
where
$$\fS_{3,n} F(x) =
\sup_{L<(1-a^{-1})n}
2^{-Ld}\chi_{B_{L+1}} *
2^{(n-L)ad /p} |T_{n-L} f_{n-L}|_{\sB}.
$$
It thus suffices to prove
\Be \label{fSn3est}
\|\fS_{3,n} F\|_p \lc 2^{-n d(\frac{1}{p_0}-\frac 1p)} \cA( p)
\|F\|_{\ell^p(L^{\wp})}
\,.
\Ee

We shall use an analytic interpolation argument and for this it  is
necessary to linearize the operator.
For any bounded linear functional $\la\in \sB^*$ we denote by
$\inn{v}{\la}$ the action of  $\la$ on  $v\in \sB$.
Let $(x,y)\to u_L(x,y)$ be any measurable function with values in $\sB^*$,
so that $ \|u_L\|_\infty\le 1$. After replacing a $\sup$ in $L$ by an $\ell^p$
norm and interchanging an integral and a summation
it then suffices to bound
$$\Big(\sum_{L<(1-a^{-1})n}\Big\|
 2^{(n-L)ad /p}
2^{-Ld}\int \chi_{B_{L+1}}(y)
 \biginn{T_{n-L} f_{n-L}(\,\cdot\, -y)}{u_L(\,\cdot\,,y)} dy
\Big\|_p^p\Big)^{1/p}
$$
by the right hand side of \eqref{fSn3est}, with a constant uniform in the
choices of the $u_L$.
In what follows we fix such a choice.

Define an analytic family
$$\cG_{L,n}^z F(x)=  2^{(n-L)ad (1-z)/p_0}
2^{-Ld}\int \chi_{B_{L+1}}(y)
 \biginn{T_{n-L} f_{n-L}(x-y)}{u_L(x,y)} dy.
$$
We then  show that for $p_0\le \widetilde{p}\le \infty$
\begin{multline}\label{Gest}
\Big(\sum_{L<(1-a^{-1})n}\big \|\cG_{L,n}^z
F\big\|_{\widetilde{p}}^{\widetilde{p}} \Big)^{1/\widetilde{p}}
 \lc 2^{-n d(\frac{1}{p_0}-\frac {1}{\widetilde{p}})}
\cA( \widetilde{p})\|F\|_{\ell^{\widetilde{p}}(L^{\widetilde \wp})}\,, \\
\quad
(1-\Re(z))(\frac 1{\wp_0},\frac 1p)+ \Re(z)
(\frac{1}{\wp_1},\frac 1\infty)=(\frac {1}{\widetilde{\wp}},\,
\frac {1}{\widetilde{p}})
\end{multline}
and the required $L^p$ estimate follows if we let $z= (1-p_0/p)$.
By Stein's theorem on analytic families of operators it suffices to show
\eqref{Gest} for $\Re(z)=0$, $\widetilde{p}=p_0$ and
$\Re(z)=1$, $\widetilde{p}=\infty$.

First, for $\widetilde{p}=p_0$, $z=\im \gamma$  we bound
\begin{align*}
&\Big(\sum_{L<(1-a^{-1})n}\big \|\cG_{L,n}^{\im \gamma} F
\big\|_{p_0}^{p_0} \Big)^{1/p_0}
\lc
\Big(\sum_{L<(1-a^{-1})n} 2^{(n-L)ad }
\big\|{T}_{n-L} f_{n-L}\big\|_{L^{p_0}(\sB)}^{p_0}\Big)^{1/p_0}
\\
&\lc A \Big(\sum_{L<n}\big\|f_{n-L}\big\|_{\wp_0}^{p_0}\Big)^{1/p_0}
\lc A\|F\|_{\ell^{p_0}(L^{\wp_0})}
\end{align*}
which is \eqref{Gest} for $\Re(z)=0$.

Now let $\widetilde{p}=\infty$, $\Re(z)=1$.
The required bound for $\cG_{L,n}^{1+\im \gamma}$ follows if we can show that
for any
 fixed $x_0$ and fixed $L<(1-a^{-1})n$
\Be \label{x0}
2^{-Ld} \chi_{B_{L+1}} *
\big| T_{n-L} f_{n-L}\big|_{\sB}(x_0)
\lc \,2^{-n d /p_0} (C_1^{d/p_0} A+B) \|F\|_{\ell^\infty(L^{\wp_1})}\,.
\Ee

Let $\cQ^*$ be a  cube of sidelength $20 d C_1 2^{(n-L)(a-1)}$
centered at $x_0$; recall the inequality   $(n-L)(a-1)>L$.
We dominate  the left hand side of \eqref{x0} by $C(\Ennear+\Enfar)$,
where
\begin{align*}
\Ennear &=
\int 2^{-Ld}\chi_{B_{L+1}}(x_0-z)
 \big| T_{n-L} [f_{n-L}\chi_{\cQ^*}](z)\big|_{\sB} \, dz \, ,
\\
\Enfar &=
\int 2^{-Ld}\chi_{B_{L+1}}(x_0-z)
 \big| T_{n-L} [f_{n-L}\chi_{\bbR^d\setminus\cQ^*}](z)\big|_{\sB} \, dz \, .
\end{align*}

By H\"older's inequality
\begin{align*}\Ennear&\lc
\Big(2^{-Ld}\int \big| T_{n-L} [f_{n-L}\chi_{ \cQ^*}](z)\big|_{\sB}^{p_0} \,
dz\Big)^{1/p_0}
\\
&\lc A 2^{-(n-L)ad/{p_0}}
2^{-Ld/p_0}\big\|f_{n-L}\chi_{ \cQ^*}\big\|_{\wp_0}
\end{align*}
and since $\|f_{n-L}\chi_{ \cQ^*}\|_{\wp_0} \lc
C_1^{d/p_0} 2^{(n-L)(a-1)d/p_0}\|f_{n-L}\|_{\wp_1} $
this yields
$$
\Ennear \lc C_1^{d/p_0}
A 2^{-nd/{p_0}}
\|F\|_{\ell^\infty(L^{\wp_1})}\,.
$$

Next observe that if $x_0-z\in B_{L+1}$ and $y\in \bbR^d\setminus \cQ^*$
then $|z-y|\ge C_1 2^{(n-L)(a-1)}$ and thus
\begin{align*}
\Enfar &\lc
\int
2^{-Ld}\chi\ci{B_{L+1}}(x_0-z)
\int_{|z-y|\ge C_1 2^{(n-L)(a-1)}}|h_{n-L}(z-y)|_\sB|f_{n-L}(y)|\,dy\, dz
\\
 &\lc 2^{-(n-L)ad/p_0}B\|f_{n-L}\|_{\wp_1}\,.
\end{align*}
Since by assumption  $aL<(a-1)n$ we get
$$\Enfar \lc
B 2^{-nd/{p_0}} \|F\|_{\ell^\infty(L^{\wp_1})}\,.$$
The estimates for $\Ennear$ and $\Enfar$ yield
\eqref{x0}. This finishes the estimation of $\fS_{3,n}(F)$ and thus concludes the proof of the theorem.

\end{proof}

\subsection*{\bf Applications}
We show how for the semigroups $U^a_t$ one can use Theorem
\ref{dyadicpieces} to prove global estimates from frequency
localized versions.

Let
 $\varphi\in C^\infty(\bbR)$ be supported in $(1/4,4)$ and not identically zero.
Define $\scr=s_{a\!,\text{cr}}$ by \Be\label{scr}\scr(\wp;p,q) :=
d\Big(\frac 1{\wp}-\frac 1p\Big)+a\Big(\frac d2-\frac d{\wp}-\frac
1q\Big) \Ee which, for a large range of parameters, turns out to be
a critical for $L^\wp_s\to L^p(\R^d;L^q(I))$ estimates; in
particular $\scr(p;p,q) = a(\frac d2-\frac d{p}-\frac 1q)$. Let
\Be\label{Gammadef} \Gamma_a(\wp;p,q) :=\sup_{R>1} R^{\scr(\wp;p,q)}
\big\|U^a \varphi \big(R^{-1}|D|\big)\big\|_{L^\wp\to
L^p(\R^d;L^q[-1/2,1/2])}\,.
\Ee
Clearly this definition depends on $\varphi$, however the finiteness of
$\Gamma_a(\wp;p,q)$ is independent of the particular $\varphi$ used.

\begin{proposition}\label{combin} Let $\wp_0,p_0,q_0\in[1,\infty]$,
$p\in(p_0,\infty)$, $q\ge  q_0$, $r\in (0,\infty)$, $\wp_0\le p_0$,  and let $I$ be a compact interval.
Suppose that $\Gamma_a(\wp_0;p_0,q_0)$
 is finite
and let $\scr$ be as in \eqref{scr}.
Assume that $1/\wp_0-1/p_0=1/\wp-1/p$.
Then:

(i)
\begin{equation}\label{dyadSchr}
\Big\|\, \Big(\sum_{k>0} \Big(\int_I |P_k U^a\!f(\,\cdot\,,t) |^q dt\Big)^{r/q}\Big)^{1/r}
\Big\|_{L^p(\mathbb{R}^d)} \lc
\Big(\sum_{k>0} 2^{ksp}\|P_k f\|_{\wp}^p\Big)^{1/p}\, ,\quad s=\scr(\wp;p,q),
\end{equation}
and
\Be\label{besovmixednormest}
\Big\|\Big(\int_I|U^af(\cdot,t)|^q\,dt\Big)^{1/q}\Big\|_{L^p(\bbR^d)}
\lc \|f\|_{B^\wp_{s,p}(\bbR^d)}, \quad s=\scr(\wp;p,q)\,.
\Ee

(ii) If $1/\wp_0-1/p_0=1/\wp_1$ then
\begin{equation}\label{dyadSchrbmo}
\Big\|\, \Big(\sum_{k>0} \Big(\int_I |P_k U^a\!f(\,\cdot\,,t) |^q dt\Big)^{r/q}\Big)^{1/r}
\Big\|_{BMO(\mathbb{R}^d)} \lc
\sup_{k>0} 2^{ks}\|P_k f\|_{\wp_1}\,,\quad s=\scr(\wp_1;\infty,q)\,.
\end{equation}

(iii) If $t\mapsto \varpi(t)$ is smooth and compactly supported
 then
\begin{equation}\label{besovschr}
\big\|\, \varpi\,U^a \!f \big\|_{L^p(\bbR^d;B^q_{\gamma,r}(\bbR))}
 \lc
\|f\|_{B^\wp_{s+a\gamma,p}(\bbR^d)}, \quad s= \scr(\wp;p,q)\,.
\end{equation}
If $f\in B^\wp_{s,p}(\bbR^d)$ with  $s= \scr(\wp;p,\infty)$
then the function $t\mapsto U^a f(x,t)$
is continuous   locally in $B^q_{1/q,r}(\bbR)$,
for almost every $x\in \bbR^d$,
 and we have the maximal inequality
\begin{equation}
\label{max}
\big\|\sup_{t\in I} |U^a \!f(\,\cdot\,,t)|  \big\|_{L^p(\bbR^d)}
 \lc
\|f\|_{B^\wp_{s,p}(\bbR^d)}, \quad s=
a(\tfrac d2-\tfrac d{\wp})+
d(\tfrac 1{\wp}-\tfrac 1p)\,.
\end{equation}

\end{proposition}
We note that the  constants implicit in \eqref{dyadSchr} and \eqref{besovschr}
depend on $p, q, q_0, I, a,d, \varpi$.

For the proof of Proposition
\ref{combin} we need  a standard imbedding result.
\begin{lemma}\label{bourg0}
Let $\wp, p\in [1,\infty]$ and $1\le q_0\le q\le\infty$. Then
$$\Gamma_a(\wp;p,q)\lc \Gamma_a(\wp;p,q_0).$$
\end{lemma}

\begin{proof} Let $h$ be in $C^1(-1/2,1/2)$. By the fundamental theorem
of calculus,
$$
|h(t)|^{q_0}\le |h(\tau)|^{q_0}+{q_0}\int_{-1/2}^{1/2} |h(y)|^{{q_0}-1}
|h'(y)|\,dy.
$$
for all $t,\tau\in (-1/2,1/2)$. Integrating in  $\tau\in (-1/2,1/2)$ and applying
H\"older's inequality yields
$$
\sup_{-1/2<t<1/2}|h(t)|^{q_0}\le \|h\|^{q_0}_{q_0}+q_0
\|h\|_{q_0}^{{q_0}-1}\|h'\|_{q_0}
$$
where the $L^q$ norms are on $(-1/2,1/2)$. Now, as $\|h\|^q_{q}\le \|h\|_{\infty}^{q-q_0}\|h\|_{q_0}^{q_0},$ we have
$$
\|h\|_{q} \le 2^{1/q}\Big(\|h\|_{q_0} + q_0^{\frac 1{q_0}-\frac 1q}
\|h\|_{q_0}^{1-\frac{1}{{q_0}}+\frac{1}{q}}\|h'\|_{q_0}^{\frac{1}{{q_0}}-\frac1q}\Big)\,.
$$
Setting $\cU_R f(x,t):=U^a[\varphi(R^{-1}|D|)f](x,t)$, for fixed $x$ we apply the displayed inequality  with
$h(t)= \cU_R f(x,t)$,
then integrate and apply H\"older's inequality in $x$ to get
$$
\|\cU_R f
\|_{L^p(L^q)}
\le 2^{\frac 1p+\frac 1q}
\Big(
\|\cU_R f\|_{L^p(L^{q_0})}+
q_0^{\frac 1{q_0}-\frac 1q}
\|\cU_R f\|_{L^p(L^{q_0})}^{\frac{1}{q_0'}+\frac{1}{q}}
\|\partial_t \cU_R f\|_{L^p(L^{q_0})}^{\frac{1}{q_0}-\frac1q}\,\Big).
$$
Now by definition $\partial_tU^a=\im(-\Delta)^{a/2}
U^a$, so that
$$\|\cU_R f\|_{L^p(L^{q_0})}
+ R^{-a}  \|\partial_t \cU_Rf\|_{L^p(L^{q_0})}
\lc\Gamma_a(\wp;p,q_0)
R^{\scr(\wp;p,q)}\|f\|_{\wp}
$$
and substituting these bounds
into  the displayed  inequality implies the assertion.
\end{proof}

\begin{proof}[Proof of Proposition
\ref{combin}]
 We can reduce to the situation where
$I=[-1/2,1/2]$ or $\varpi\in C^\infty_c((-1/2,1/2))$, by a
change of variables argument.
By Lemma
\ref{bourg0} we may assume $q_0=q$.

To  prove \eqref{dyadSchr} we
 apply Theorem~\ref{dyadicpieces}.
Let $\rho$ be a radial $C^\infty(\R^d)$-function which is compactly supported in the
 ball of  radius $1/2$ centered at $0$, with the property that $\widehat \rho$
is positive on $\supp \chi(|\cdot|)$. Denote by $R_k$ the operator of convolution with $2^{kd}\rho (2^k\,\cdot\,)$.
Let $L_k=  \varphi(2^{-k}|D|)$ where
$\varphi$ is chosen so that $\varphi(|\cdot|)\widehat \rho  =1$ on $\supp
\chi(|\cdot|)$.
Thus $R_kL_k P_k=P_k$.

Now let $\sB= L^q[-1/2,1/2]$
and let
$T_k f(x,t)= 2^{k(a-1)d(\frac 1{\wp_0}-\frac 1{p_0})}
2^{-ka(\frac d2-\frac 1q)} L_kU^a f(x,t)$.
Then the hypothesis that $\Gamma_a(\wp_0;p_0,q_0)$ is finite
implies
$$A:=\sup_{k>0} 2^{k ad/{p_0}}\|T_k\|_{L^{\wp_0}\to L^{p_0}(\sB)}<\infty.$$
For fixed $t$ let $h_k^t$ be the convolution kernel for $T_k$ (at fixed time $t$); it can be written as
\begin{equation*}
h^t_k(x)=
2^{k(1-a)d(\frac 1{\wp_0}-\frac 1{p_0})}
2^{-ka(\frac d2-\frac 1q)} 2^{kd} (2\pi)^{-d} \int_{\bbR^d} \varphi(|\xi|)
e^{\im(2^k \inn x\xi+2^{a k}t|\xi|^a)} d\xi.
\end{equation*}
An $N$-fold integration by parts yields
\begin{equation*}
2^{k(a-1)d(\frac 1{\wp_0}-\frac 1{p_0})} 2^{ka(\frac d2-\frac 1q)}
|h^t_k(x)|\le C_{N} 2^{k(d-N)}|x|^{-N},\quad |x|\ge  2^{(a-1)k+4},
\quad t\in [0,1]\,,
\end{equation*}
and thus condition  \eqref{smallness} is satisfied with $C_1= 2^5$.
 Thus, by Theorem
\ref{dyadicpieces}  we obtain the inequality \Be\label{Schrvectval}
\Big\|\,\Big( \sum_{k>0} 2^{k a\frac dp r} \|R_k{T}_{k}
f_{k}\|_{L^q[-1/2,1/2]}^r\Big)^{1/r}
 \,\Big\|_{p} \lc
\Big(\sum_{k>0} \|
f_{k}\|_{\wp}^p\Big)^{1/p}\,.
\Ee
Notice that, in view of $1/\wp_0-1/p_0=1/\wp-1/p$,
$$2^{kad/p} T_k= 2^{-k\scr(\wp;p,q)} L_kU.$$
Thus if we apply \eqref{Schrvectval}
with
$f_k=2^{k\scr(\wp;p,q)} P_k f$,
 then \eqref{dyadSchr}
follows.
 The assertion \eqref{dyadSchrbmo} is obtained in the same way.

We now  need  to show how to obtain   \eqref{besovmixednormest}
 from \eqref{dyadSchr}.
The right hand side in \eqref{dyadSchr} is just the $B^\wp_{s,p}$-norm of $f$.
We also have for $1<q,p<\infty$
\Be\label{LiPa}
\Big\|\Big(\int_I |G(\cdot,t)|^q dt\Big)^{1/q}\Big\|_p
\lc
\Big\|\Big(\int_I
\Big[\Big(\sum_{k\ge 0} |P_k G(\cdot,t)|^2\Big)^{1/2}\Big]^q
dt\Big)^{1/q}\Big\|_p\,
\Ee
which we apply for $G(x,t)=Uf(x,t)$.
For $p=q$ this
 is just a consequence of  the standard Littlewood-Paley
inequality (after interchanging the $x$ and the $t$ integral).
By Calder\'on-Zygmund theory the estimate also holds for $1<p<q$
(and is obtained by interpolation with a weak-type (1,1) estimate for
$L^q(\ell^2)$-valued functions). If we dualize a  similar reasoning
yields the case $q>p$.
Inequalities \eqref{dyadSchr},  \eqref{LiPa} easily imply \eqref{besovmixednormest}.

Now consider a  standard inhomogeneous Littlewood-Paley
decomposition $\{\cP_k\}_{k=0}^\infty$ on $L^p(\bbR^d)$   so that $\cP_k=P_k$ for $k>0$  and where $\cP_0$ localizes to frequencies with $|\xi|\le 2$.
For the  estimation of $\cP_0 U^a$  standard multiplier arguments
 apply.
We also   need to consider a similar inhomogeneous
Littlewood-Paley decomposition in the $t$ variable, which we denote by
$\{\cL_j\}_{j=0}^\infty$.
Then inequality \eqref{besovschr} can be rewritten as
\Be \label{besov-vv}
\Big\| \Big(\sum_{j=0}^\infty 2^{j\gamma r}
\big\|\cL_j [\varpi U^a\! f]\big\|_{L^q(\bbR)}^r\Big)^{1/r}
\Big\|_{L^p(\bbR^d)} \lc
\Big(\sum_k 2^{k(s+a\gamma)p}\big\|\cP_k f\big\|_{L^\wp(\bbR^d)}^p\Big)^{1/p}.
\Ee
We claim that there is a constant $M$ for which
\Be\label{junk}
\big\|\cL_j [\varpi U^a \cP_k g]
\big\|_{L^p(L^q)}
\le C_N \min \{2^{-jN}, 2^{-ka N}\} \|g\|_\wp \quad \text{ whenever } \ |ka-j|\ge M,
\Ee
so that for the essential terms $k$ and $j$ are coupled via
$|ka-j|\le M$.
This would mean that a $t$ derivative of order $\alpha$
could be traded with an $x$ derivative of order $a\alpha$, so that \eqref{besov-vv} would follow from \eqref{Schrvectval}.
Thus it remains to prove \eqref{junk}.
Note that for $k>0$, $j>0$, the convolution kernel of $g\to\cL_j [\varpi U^a \cP_k g](\,\cdot\,,t)$ can be written as
\begin{equation*}
\frac{1}{(2\pi)^{d+1}} \iint  \Big\{\int \varpi(s)\,e^{\im s
(|\xi|^a-\tau)} \chi_1(2^{-j}|\tau|) \chi_1(2^{-k}|\xi|) \, ds\Big\}
\,  e^{\im(\inn x\xi +t\tau)} d\tau\, d\xi
\end{equation*}
and similar formulas hold if either $k=0$ or $j=0$.
One checks that if $|ka-j|\gg 1$ then for $\xi$ and $\tau$ in the support of the indicated cutoff functions the inequality
$||\xi|^a-\tau| \ge c\max \{|\xi|^a, |\tau|\}$
holds. We perform $N+d+1$
integration by parts in $s$. For $t$ large, we  follow this
 by integrations by parts in $\tau$. This easily
yields \eqref{junk}.

The final assertions of the proposition
are a consequence of the fact that for $r\le 1$ the space
$B^q_{1/q, r}(\bbR)$ is imbedded in the space of bounded continuous functions.
\end{proof}

\end{document}